\documentclass[leqno,11pt]{amsart}
\usepackage{amssymb, amsmath}

\usepackage{color,ulem,textcomp}

\numberwithin{equation}{section}
\usepackage{enumerate}
\usepackage{graphicx}
\usepackage{cancel}

\theoremstyle{definition}
\newtheorem{rmk}{Remark}[section]
\theoremstyle{plain}
\newtheorem{thm}{Theorem}
\newtheorem{Prop}[rmk]{Proposition}
\newtheorem{lem}[rmk]{Lemma}

\def\bC {\mathbb{C}}

\def\R {\mathbb{R}}

\def\T {\mathbb{T}}
\def\Z {\mathbb{Z}}

\def\cC {\mathcal{C}}

\def\e {{\varepsilon}}
\def\eps {{\varepsilon}}

\def\pa {{\partial}}

\def\na {{\nabla}}

\newcommand{\Div}{\operatorname{div}}
\newcommand{\rot}{\operatorname{rot}}

\newcommand{\ba}{\begin{aligned}}
\newcommand{\ea}{\end{aligned}}
\newcommand{\be}{\begin{equation}}
\newcommand{\ee}{\end{equation}}

\newcommand{\etah}{\eta_{\rm h}}

\newcommand{\curl} {\mbox{curl} \, }
\newcommand{\FL} {\mbox{F}_L}
\newcommand{\FLz} {\mbox{F}_{L,3}}
\newcommand{\FLh} {\mbox{F}_{L,{\rm h}}}
\newcommand{\xh} {x_{\rm h}}

\newcommand{\PP}{   \mathbb {P}}

\numberwithin{equation}{section}

\begin{document}

	\title[Wellposedness of linearized Taylor equations]{Wellposedness of linearized Taylor equations \\ in magnetohydrodynamics}
	\author{Isabelle Gallagher, David G\'erard-Varet}
	\date{\today}
		
		\address[Isabelle Gallagher]{Universit\'e Paris Diderot, Sorbonne Paris Cit\'e, Institut de Math\'ematiques de Jussieu-Paris Rive Gauche, UMR 7586, F- 75205 Paris, France}
		\address[David G\'erard-Varet]{Universit\'e Paris Diderot, Sorbonne Paris Cit\'e, Institut de Math\'ematiques de Jussieu-Paris Rive Gauche, UMR 7586, F- 75205 Paris, France}
		
\begin{abstract}
This paper is a first step in the study of the so-called Taylor model, introduced by J.B. Taylor in \cite{Taylor}. This system of nonlinear PDE's is derived from the viscous incompressible MHD equations, through an asymptotics relevant to the Earth's magnetic field. We consider here a simple class of linearizations of the Taylor model, for which we show well-posedness.

 \end{abstract}
	
	\maketitle
	
	\section{Presentation of the model and main result}
	\subsection{Introduction}
	The concern of this paper is the so-called Taylor model, derived by J.B. Taylor in 1963. The general motivation behind this model is the understanding of the dynamo effect in the Earth.  By dynamo effect, we mean the generation of  magnetic energy by the flow of   liquid iron in the Earth's core. This dynamical process has been recognized since the first half of the twentieth century, and sustains the magnetic field of 
the Earth, despite Joule dissipation. We refer to \cite{Gilbert,Dormy_Soward} for an introduction to the subject. 

A standard model in dynamo theory is the so-called incompressible MHD system, which is obtained after  coupling and simplifying  the Navier-Stokes and Maxwell equations (see~\cite{Gilbert}, \cite{Arsenio}). The resulting system reads 
\begin{equation} \label{MHD}
\begin{aligned}
&  \rho (\pa_t u + u \cdot \na u) + \na p + \rho \Omega_0 e \times u  - \mu \Delta u =  \mu_0^{-1} \curl B \times B\\Â 
& \pa_t B = \curl (u \times B) + \eta \Delta B  \\
& \Div u = \Div B = 0 \, .  
\end{aligned}
\end{equation}
The first line corresponds to the Navier-Stokes equation, that describes the evolution of the fluid velocity $u$ and pressure $p$. The density $\rho$ and viscosity $\mu$ are constant.  The equation is written in a frame rotating with the Earth, which is responsible for the Coriolis forcing term $\rho \Omega_0 e \times u$, with $\Omega_0$ the angular speed of the Earth and $e = e_3$ is taken as the rotation axis. Finally, as one is describing a conducting fluid, one must take into account the Laplace force  $\mu_0^{-1} \curl B \times B$ exerted by the magnetic field $B$ on the fluid ions, with  $\mu_0$ the magnetic permeability constant. 

The second line is the so-called induction equation, that describes the evolution of the magnetic field. 
It can be written $\pa_t B = \curl E$, where the electric field $E  = u \times B - \eta \rot B$ is deduced from   Ohm's law in a moving medium (see \cite{Gilbert} for details). 
Finally, the divergence free constraints on $u$ and $B$ correspond to the incompressibility of the fluid and the absence of magnetic monopole respectively.

With regards to the dynamo problem, the MHD system has been the matter of many works, see  \cite{Roberts, Arnold, Friedlander1,GV,GVRousset} among many. Most of them focus on linear studies : namely, by linearizing  \eqref{MHD}  around $(u = u(x), B = 0)$, one is left with 
\begin{equation} \label{kinematic}
 \pa_t b = \curl(u \times b) + \eta \Delta b 
\end{equation}
where $u$ is given. In other words, the retroaction of the magnetic field on the fluid is neglected, and one tries to determine which fluid flows allow for the growth of  magnetic perturbations~$b$. This amounts to establishing the existence of unstable spectrum for the operator at the right-hand side of \eqref{kinematic}. 
However, this spectral problem turns out to be difficult. Roughly, to be a dynamo field, $u$ must exhibit some kind of complexity. For instance, if $u$ has too many symmetries, there is only stable spectrum : this is the point of several antidynamo theorems, see \cite{Cowling,Arnold}. Also, if one looks for fast dynamos, meaning with a lower bound on the growth rate independent of the magnetic diffusivity $\eta$, then the field $u$ must exhibit some lagrangian chaos: we refer to \cite{Vishik,Childress}  for more on fast dynamos. 

Hence, a good understading of the Earth's magnetic field through explicit analytical calculations seems out of reach.  Unfortunately, numerical simulation of the MHD system is  also a very challenging problem, due to the presence of many small parameters : in a dimensionless form, \eqref{MHD} becomes  \begin{equation} \label{MHD2} 
\begin{aligned} 
& \pa_t u + u \cdot \na u  + \frac{\na p}{\eps} + \frac{e \times u}{\eps} - \frac{E}{\eps} \Delta u = \frac{\Lambda}{\eps \theta} \curl B \times B \\
& \pa_t B = \curl (u \times B) + \frac{1}{\theta} \Delta B \\
& \Div u  = \Div B = 0 \, . 
\end{aligned}
\end{equation}
The dimensionless parameters $\eps, E, \Lambda$ and  $\theta$ are the Rossby, Ekman, Elsasser and magnetic Reynolds numbers respectively. Typical values for the Earth's core are  
 $$ \eps \sim 10^{-7}\,, \quad \Lambda = O(1)\,, \quad \eps \theta \sim 10^{-4} \, , \quad E \sim 10^{-15} \, . $$
We refer to \cite{DDG} for more.   

Due to these small values and underlying small-scale phenomena, a direct computation of the solution is not possible. It is thus tempting to simplify \eqref{MHD2}, notably neglecting the inertia and viscous dissipation of the fluid. Proceeding formally, we obtain 
\begin{equation} \label{Taylor_parameters}  
\begin{aligned} 
 & e \times u + \na p = \frac{\Lambda}{\theta} \curl B \times B  \\
 & \pa_t B = \curl (u \times B) + \frac{1}{\theta} \Delta B \\
 & \Div u = \Div B = 0 \, .
\end{aligned}
\end{equation}  
This system was introduced formally and briefly discussed by  J.B. Taylor in  \cite{Taylor}. Its mathematical analysis is the subject of the present paper. Let us stress that other "degeneracies" of the MHD system have been recently considered, in link  to the magnetic relaxation problem introduced by  K. Moffatt. We refer  to \cite{MRR}, \cite{Brenier}. 

\subsection{The Taylor model}
Our long term goal is to understand better and possibly justify the asymptotics that leads from \eqref{MHD2} to \eqref{Taylor_parameters}. 
 In geophysical contexts, a huge litterature has been devoted to evolution equations with a linear skew-symmetric penalization \cite{CDGG, Feireisl}. The peculiarity of the present problem is its genuine nonlinear character, originating in  the penalization of the nonlinear Laplace force $\FL := \curl B \times B$ in \eqref{MHD2}. We shall only discuss here the limit Taylor system. For simplicity, we normalize all constants to $1$ and thus consider
 \begin{equation} \label{Taylor}  
\begin{aligned} 
 & e \times u + \na p =  \curl B \times B  \\
 & \pa_t B = \curl (u \times B) + \Delta B \\
 & \Div u = \Div B = 0 \, .
\end{aligned}
\end{equation}  
 We assume that this system is set in  a domain $\Omega$  with an impermeable boundary: $u\cdot n \vert_{\pa \Omega} = 0$ (we do not discuss the boundary conditions on $B$ for the time being).

 Let us start with general comments on the first equation (\ref{Taylor}a).  
 Time is only a parameter there, so that we omit it temporarily from notations.  This equation involves naturally the Coriolis operator   $ \cC u := \PP(e \times u)$, where $\:   \PP:=\mbox{Id} - \nabla \Delta^{-1} \mbox{div}$ 
 is the Leray projector onto divergence free vector fields. It defines a skew-symmetric operator over 
 the space $L^2_\sigma(\Omega)$  of $L^2$ divergence free fields tangent at $\pa \Omega$. The balance equation (\ref{Taylor}a)  implies the {\it Taylor constraint}:
 \begin{equation} \label{Taylor_constraint}
 \PP \FL \in  \mbox{Range}(\cC) \, .
 \end{equation}
 Under this constraint, for a given $B$ such that $\PP \FL$ belongs to $L^2_\sigma(\Omega)$,  any solution $u$ of  (\ref{Taylor}a) can be written 
 $u = u_m + u_g$, where 
 \begin{itemize}
 \item 
 $u_m$ belongs to the orthogonal of $\ker \cC$, and is uniquely determined by $\PP \FL$. In particular, it can be expressed in terms of $B$, possibly in an implicit way. Such a field $u_m$ will be called {\it magnetostrophic}. 
\item 
$u_g$ is any field in $\ker \cC$. It satisfies
\begin{equation} \label{geostrophicbalance}
e \times u + \na p = 0 \, , \quad \Div u = 0 \, , \quad u\cdot n\vert_{\pa \Omega} = 0 \, . 
\end{equation}
 Such a field $u_g$ will be called {\it geostrophic}. 
 \end{itemize}
 Note that, by skew-symmetry of the Coriolis operator and by the Taylor constraint \eqref{Taylor_constraint},~$\PP \FL$ must be orthogonal to $\ker \cC$ :
 \begin{equation} \label{Taylor_constraint2}
 \int_\Omega \PP \FL \cdot u_g  = \int_{\Omega} \FL \cdot u_g = 0 \, , \quad \mbox{ for all geostrophic fields} \: u_g \, .
 \end{equation}
 Inserting the decomposition $u = u_m + u_g$ into the induction equation (\ref{Taylor}b), we get
 \begin{equation} \label{inductionmg}
  \pa_t B = \curl(u_m \times B) + \curl(u_g \times B) + \Delta B \, .
  \end{equation}
   As $u_m = u_m(B)$, the first  term at the right hand-side can be seen a nonlinear functional of~$B$. The other term is more delicate, as the geostrophic field $u_g$  is {\it a priori} not determined. The idea is that the term $\curl(u_g \times B)$ should be a kind of Lagrange multiplier associated with the Taylor constraint \eqref{Taylor_constraint}. From this point of view, a parallel can be drawn with the incompressible Navier-Stokes: the term $\curl(u_g \times B)$ would correspond to the pressure gradient, whereas the Taylor constraint would correspond to the incompressibility condition. 
   
Let us for instance consider the case of the ball  $\Omega = B(0,1)$, discussed in \cite{Taylor}. The geostrophic fields $u_g$  have a simple characterization. It can be shown that   they are exactly those of the form $ u_g = (0,U_\theta(r), 0)$  in cylindrical coordinates $(r,\theta,z)$. Hence, condition \eqref{Taylor_constraint2} amounts to
\begin{equation} \label{constraint_sphere}
 \int (\FL)_\theta \, dz d\theta = 0 
 \end{equation}
along any cylinder $r = r_0$ in cylindrical coordinates.  Moreover, if $\FL$ is regular enough, \eqref{constraint_sphere} is equivalent to the original Taylor constraint \eqref{Taylor_constraint}.  See \cite{Taylor} for more details. 

Again, the term $\curl(u_g \times B)$ should allow to preserve \eqref{constraint_sphere} through time. We can write 
\begin{align*}
 0 & = \pa_t   \int (\FL)_\theta dz d\theta = \pa_t   \int (\curl B \times B)_{\theta}   dz d\theta \\
& = \int \left( \curl \pa_t B \times B + \curl B \times \pa_t B \right)_{\theta} dz d\theta 
\end{align*}
 and substitute for $\pa_t B$ using equation \eqref{inductionmg}. This formal manipulation is performed by Taylor in \cite{Taylor}. He finds an elliptic equation of second order on $U_\theta$, whose coefficients and source depend on $B$ (the source is explicit in terms of $B$ and $u_m$).  Such an equation can be seen as an analogue of the Poisson equation on pressure which is derived from the incompressible  Navier-Stokes system.  Following \cite{Taylor}, one may hope to invert this elliptic equation and express  in this way $u_g$ in terms of $B$. Eventually the evolution equation \eqref{inductionmg}  may make sense !

Of course, above computations and remarks lack mathematical justification. The present paper aims at taking a little step forward. 
  
 \subsection{The Taylor model in the torus}    \label{paragraph_Taylor_torus}
To avoid technicalities due to boundaries, we consider the Taylor system \eqref{Taylor} in the torus $\T^3$. The space $L^2_\sigma(\Omega)$ is now replaced by the space~$L^2_\sigma(\T^3)$  of $L^2$ divergence-free fields with zero mean over $\T^3$.  The Taylor constraint can  be made explicit in this setting. We assume that $\FL$ is smooth enough, and take the curl of equation (\ref{Taylor}a).  We find
\begin{equation}\label{taylor}
\partial_3 u = \curl  \, \FL\, ,  
\end{equation}
which is solvable if and only if: 
\begin{equation} \label{constraint_torus}
\int_{\T} \curl \, \FL (\cdot,x_3) \,  dx_3 = 0\, .
\end{equation}
 Under this condition, $\curl \FL$ has a unique antiderivative with zero mean in $x_3$, that is 
\begin{equation} \label{formula_um}
 u_m(\cdot,x_3) = \int_0^{x_3}\curl \,\FL (\cdot,y_3) \,  dy_3 -  \int_{\T}\int_0^{x_3}\curl \,\FL (\cdot,y_3) \,  dy_3 \, .
\end{equation}
 This field $u_m$ is a solution of \eqref{taylor}, but also of equation (\ref{Taylor}a). Indeed, it follows from~\eqref{taylor} that $\pa_3 \text{div} \, u_m = 0$, and as $u_m$ has zero mean in $x_3$, $\text{div} \, u_m = 0$. Hence,  the relation $ \pa_z u_m = \curl \FL$ can be written $\curl (e \times u_m) = \curl \FL$, which is the same as (\ref{Taylor}a). More generally, any solution  of (\ref{Taylor}a) is of the form $u = u_m + u_g$, where $u_g$ is any element in the kernel of the Coriolis operator.  It is well-known that these elements are the  $u_g = u_g(t,\xh)$, with~$\xh:=(x_1,x_2)$,  satisfying  
\begin{equation}\label{geostrophic}
\Div_{\rm h} u_{g,{\rm h}}  :=\partial_1 u_{g,1} + \partial_2 u_{g,2} = 0   \quad \mbox{and} \quad  \int_{\T^2}u_g(\cdot,\xh)\, d\xh = 0 \, .
\end{equation}
Clearly, $u_m$ is orthogonal to any field $u_g$ of the previous kind in $L^2_\sigma(\T^3)$. Hence, according to the terminology of the previous paragraph, we have identified the Taylor constraint \eqref{constraint_torus}, the magnetostrophic field $u_m$, and the geostrophic fields $u_g$.   A straightforward computation shows that \eqref{constraint_torus} is equivalent to  
\begin{equation} \label{variation_taylor} 
\int_\T \FLz(\cdot,x_3) dx_3 = 0 \, , \quad \mathbb{P}_{\rm h}    \int_\T \FLh(\cdot,x_3) dx_3 = 0 
\end{equation} 
where $\mathbb{P}_{\rm h} $ is the two-dimensional  Leray projector over $L^2_\sigma(\T^2)$. By the identity $\curl B \times B = B \cdot \na B + \frac{1}{2} \na |B|^2$,  it can also be written
\begin{equation} \label{variation_taylor_2} 
\int_{\T} B \cdot \na_{\rm h}  B_3  \, dx_3 = 0 \,, \quad  \mathbb{P}_{\rm h}   \int_{\T} B \cdot \na_{\rm h}  B_{\rm h}  \,  dx_3 = 0 \, .  
\end{equation} 
It is equivalent  to  \eqref{Taylor_constraint2} as well.

 One still needs to make sense of  \eqref{inductionmg}, notably of the Lagrange multiplier $\curl(u_g \times B)$. We follow the approach initiated by Taylor in the case of the ball. Assuming that $B,u_m, u_g$ are smooth enough and satisfy  \eqref{inductionmg}, we derive an evolution equation for $B \cdot \na B$. We write 
\begin{align*}
\pa_t (B \cdot \na B) & = B \cdot \na \pa_t B + \pa_t B \cdot \na B \\
 & = B \cdot \na (\curl(u_g \times B)) +  (\curl(u_g \times B) ) \cdot \na B + F
\end{align*} 
where 
$$ F = B \cdot \na (\curl(u_m \times B) + \Delta B) +  (\curl(u_m \times B) + \Delta B) \cdot \na B $$
 is a nonlinear functional of $u_m$ and $B$, that can be seen by \eqref{formula_um} as a nonlinear functional of $B$ alone.   Using the relation 
 \begin{equation} \label{curl_identity}
 \curl (a \times b) = b \cdot \na a - a \cdot \na b \,,
 \end{equation}
  valid for  all divergence-free vector fields $a,b$, and after a few simplifications, we obtain 
\begin{equation}
 \pa_t (B \cdot \na B) + u_g \cdot \na (B \cdot \na B) = (B \cdot \na)^2 u_g + F \,.
\end{equation}
The role of $u_g$ is to preserve the Taylor constraint through time. In the form  \eqref{variation_taylor_2}, this amounts to the system of equations 
$$ \left( \begin{smallmatrix} \mathbb{P}_{\rm h} &  \\ & 1 \end{smallmatrix}  \right) \left( - \int_\T (B \cdot \na)^2 dx_3 \, u_g \: + \:  u_g \cdot \int_\T \na (B \cdot \na B) dx_3 \right) =   \left( \begin{smallmatrix} \mathbb{P}_{\rm h} & \\ & 1 \end{smallmatrix}  \right)  \int_\T F dx_3  $$ 
which splits into 
\begin{equation} \label{Stokes_ug} 
  - \int_\T (B \cdot \na)^2 dx_3 \, u_{g,{\rm h}} \: + \:  u_{g,{\rm h}} \cdot\na_{\rm h}  \int_\T \na (B \cdot \na B_{\rm h}) dx_3 +\na_{\rm h} p  = \int_\T F_{\rm h} dx_3 \, , \quad \text{div} \, u_{g,{\rm h}} = 0 
  \end{equation} 
for some~$p = p(x_{\rm h})$  and 
  \begin{equation} \label{Poisson_ug}
  - \int_\T (B \cdot \na)^2 dx_3 \, u_{g,3} \:   =  \: \int_\T F_3 dx_3 \,.   
  \end{equation} 
This set of equations, where $t$ is only a parameter, can be seen  an analogue of the Poisson equation for the pressure in Navier-Stokes, or an analogue of the second order equation derived by Taylor in the case of the ball, {\it cf} the previous paragraph. Roughly, the system~\eqref{Stokes_ug}, that is satisfied by the horizontal part of the geostrophic field $u_{g,{\rm h}}$, looks like a Stokes equation, whereas the equation \eqref{Poisson_ug} satisfied by $u_{g,3}$ looks like a Poisson equation. The main difference is that the usual two-dimensional Laplacian operator is replaced by~ $\displaystyle \int_\T (B \cdot \na)^2 \, dx_3$. Moreover, the Stokes-like equation contains a zero order term.  
This  makes unclear the optimal conditions for which these equations are well-posed. Still, we can state the following result.
\begin{Prop} \label{solvability_ug}
Let $B = B(x)$ be smooth and divergence-free over $\T^3$. Assume that 
\begin{equation} \label{assumptionCS}
\text{ for all $x_{\rm h} \in \T^2$,  \: $B_{\rm h}(x_{\rm h},\cdot)$ has non-constant direction.}  
 \end{equation}
 Then, for  $\displaystyle \sup_{x_{\rm h} \in \T^2} \left| \int_\T  (B \cdot \na B_{\rm h})(x) dx_3 \right|$  small enough, and   for any smooth $F = F(x)$ with zero mean over $\T^3$,  equations \eqref{Stokes_ug} and \eqref{Poisson_ug} have  unique smooth solutions $u_{g,{\rm h}}$ and $u_{g,3}$ with zero mean over $\T^2$.  
 \end{Prop}
 \begin{proof} We first consider  \eqref{Stokes_ug}. One can associate to it the variational formulation 
 \begin{equation} \label{VF_Stokes_ug}
 \begin{aligned}
  \int_{\T^2} \int_\T  \left(B_{\rm h} \cdot\na_{\rm h} u_g \right)   \cdot \left(B_{\rm h} \cdot\na_{\rm h} \varphi \right)   d x_3 \, d x_{\rm h}  \: & + \:  \int_{\T^2} \int_\T  \left( B \cdot \na B_{\rm h} \right)    \, dx_3  \, \cdot \,  (u_g \cdot\na_{\rm h} \varphi) \, dx_{\rm h} \\ 
  & = \int_{\T^2} \int_\T F_h \, dx_3 \, \cdot  \varphi \, dx_{\rm h} 
 \end{aligned}
 \end{equation}
 for all  $\varphi$ in $H^1_\sigma(\T^2)$, that is the set of 2d solenoidal  vector fields in $H^1(\T^2)^2$ with zero mean. 
 Assumption \eqref{assumptionCS} is equivalent to the strict Cauchy-Schwarz inequality 
\begin{equation} \label{strictCS}
|\int_\T (B_1 B_2)(\cdot,x_3) dx_3 |^2 \: < \:  \int_\T |B_1(\cdot,x_3)|^2 dx_3 \, \int_\T |B_2(\cdot,x_3)|^2 dx_3 \, .  
\end{equation}
uniformly in $x_{\rm h}$. It follows that  for any $f = f(x_{\rm h})$ with zero mean 
$$ \int_{\T^2} \int_\T | B_{\rm h} \cdot\na_{\rm h} f |^2  \, dx_3 \, dx_{\rm h} \ge c  \int_{\T^2}  |\na_{\rm h} f |^2  \, dx_{\rm h}  $$ 
so that the first term in \eqref{VF_Stokes_ug} is coercive over $H^1_\sigma(\T^2)$. 
Under the smallness assumption of the proposition, the whole left hand-side of \eqref{VF_Stokes_ug} is coercive, which yields a unique solution~$u_{g,{\rm h}}$ by the Lax Milgram lemma. The smoothness of $u_{g,{\rm h}}$ follows from classical elliptic regularity, as the principal symbol of the operator $\int_\T (B \cdot \na)^2 \, dx_3$ is uniformly elliptic under \eqref{strictCS}. 

The case of \eqref{Poisson_ug} is similar and easier:  the existence of a unique smooth solution $u_{g,3}$ is obtained under the single assumption \eqref{assumptionCS}, as there is no zero order term. 
 \end{proof}
 Proposition \ref{solvability_ug} opens the way towards a local well-posedness result of the Taylor system\eqref{Taylor}. Indeed, if the initial datum $B_0$ is smooth, satisfies \eqref{assumptionCS} and if 
  $$\text{$\sup_{x_{\rm h} \in \T^2} \left| \int_\T (B_0 \cdot \na B_{0,{\rm h}})(x) dx_3 \right|$ is small enough}$$
 (for instance if it is zero), then the right-hand side of \eqref{inductionmg} is  well-defined and smooth at initial time, and one may hope to solve the equation at least for short time.   Note furthermore  that~\eqref{Taylor} is (formally) dissipative: we find
 \begin{align*}
  \frac{1}{2} \frac{d}{dt} \| B \|_{L^2(\T^3)}^2 +  \| \na B \|^2_{L^2(\T^3)} & = - \int_{\T^3}  \curl(u \times B) \cdot B  \\
& = \int_{\T^3} u \cdot (\curl B \times B)  = \int_{\T^3} u \cdot (e \times u + \na p) = 0 \, . 
 \end{align*}
Still, this energy decay is not enough to build strong solutions, because it does not provide a control of higher order derivatives. The current paper, devoted to a linearized analysis of~\eqref{Taylor}, can be  seen as a first step in the study of these derivatives.

   \subsection{Linearization and statement of the main result}
From now on,  we study the well-posedness of simple linearizations of the Taylor model in $\T^3$. With Proposition \ref{solvability_ug} in mind, we consider reference states of the form~$(u = 0, B)$ with 
 \begin{equation}\label{formofB}
   B(x):= (B_1(x_3), B_2(x_3),0) \,
   \end{equation}
   where~  $B$ has zero mean and  non-constant direction, meaning:  
\begin{equation} \label{nonconstant}
\forall \eta \in \mathbb{S}^1, \quad \| B_{\rm h} \cdot \eta \|_{L^2(\T)} > 0 \, .  
\end{equation}  
This last assumption is made coherently with Proposition \ref{solvability_ug}. Note that  $\displaystyle \int_\T B \cdot \na B_{\rm h} = 0$, so that $B$ also satisfies the smallness assumption of the proposition. Let us mention that the couple~$(u=0,B)$ is not a solution of the source free Taylor model \eqref{Taylor}: one should add a forcing term $f = - \Delta B$ at the right hand-side of (\ref{Taylor}b). But this is a usual approximation, reminiscent of the study of shear flow stability in fluid dynamics.

 The linearized system is then
  \begin{equation}
  \label{linear}
  \begin{aligned}
 & e \times u + \na p = \curl b \times B + \curl B \times b \\
  & \pa_t b = \curl (u \times B) + \Delta b \\
  & \Div u = \Div b = 0
 \end{aligned}
  \end{equation}
 where $b$ is now the magnetic perturbation. The main result of this paper is the following well-posedness result: 
 \begin{thm}\label{mainthm}
Assume that $B$ is a zero-mean smooth function of the form \eqref{formofB}, satisfying assumption \eqref{nonconstant}.  If $b_0$ belongs to~$L^2_\sigma(\T^3)$, system \eqref{linear} has a unique solution 
$(u,b)$ such that 
$$ b \in C(\R_+, L^2_\sigma(\T^3)) \cap L^2_{loc}(\R_+, H^1_\sigma(\T^3)), \quad u \in L^2_{loc}(\R_+, H^{-1}_\sigma(\T^3)), $$
satisfying for  some constant~$C$ and for all $t \ge 0$: 
\begin{equation} \label{theresult}
\|b(t)\|^2_{L^2}  + \int_0^t \| \na b(s) \|_{L^2}^2 ds \leq C \| b_0 \|^2_{L^2} \exp (Ct ) \, .
\end{equation}
\end{thm}

%

\medskip
The rest of the paper is dedicated to the proof of this theorem. As usual, the keypoint is to establish an {\it a priori} estimate of type \eqref{theresult} for smooth solutions of \eqref{linear}.  The existence and uniqueness of a solution follows then from standard arguments. But  the derivation of this {\it a priori} bound is far from obvious. The difficulty comes  from the so-called induction term~$\curl(u \times B)$ in (\ref{linear}b). As will be seen below, one can express $u = u[b]$ as a linear operator in $b$, second order in  variables~ $(x_1,x_2)$. It follows that one can write~$\curl(u \times B)  $ as~$\curl(u \times B) = L_B b$, where $L_B$ is a third order operator in $(x_1, x_2)$. The principal part of this operator $L_B$ is shown to be skew-symmetric, but second order terms remain:  one may find fields $b = b(x)$ such that 
$$ \int_{\T^3} L_B b  \cdot  b  \ge c_B \| \na b \|_{L^2}^2$$
where roughly, the constant $c_B > 0$ may grow with the amplitude of $B$. In particular, if $B$ is large, this term can not be absorbed in a standard energy estimate by the term coming from the laplacian in $b$.      
Hence, the linear system \eqref{linear} may be ill-posed, with growth similar to the one of the  backward heat equation.  

The point of the paper is to show that such instability  does not occur. It is based on a careful normal form argument,  annihilating the second order symmetric part by the third order skew-symmetric one.   

\section{Preliminaries and reductions}
\subsection{Computation of the linear Taylor constraint and the magnetostrophic field}\label{computationmagnetostropic}
In order to prove Theorem~\ref{mainthm}, we first need  to compute~$\curl (u\times B)$ in terms 
of~$b$. The approach is the same as in the nonlinear analysis of Paragraph \ref{paragraph_Taylor_torus}. We first focus on equation (\ref{linear}a), which amounts~to 
$$ \pa_3 u = \curl \left( \curl b \times B + \curl B \times b \right) $$
and	is solvable if and only if 
\begin{equation} \label{linear_constraint}
\int_\T \curl \left( \curl b \times B + \curl B \times b \right) dx_3 = 0 \, .  
\end{equation}
Under this  constraint,  (\ref{linear}a) has a unique solution $u_m$ with zero mean in $x_3$. 
Let 
\begin{equation}
 \pa_3^{-1} \:  :=  \: {\mathcal P} \int_0^{x_3} {\mathcal P} \, , \quad \text{where} \quad {\mathcal P} \, f \: :=  \: f - \int_\T f dx_3  
 \end{equation}
is the projection onto functions with zero mean in $x_3$.   Then: 
\begin{equation}\label{formulaum} 
 u_m \: = \:  u_{m,1}[b] + u_{m,2}[b] \, ,  
\end{equation}
with 
$$  u_{m,1}[b] \:  := \:  \pa_3^{-1} \curl \left( \curl b \times B \right), \quad u_{m,2}[b] \: := \: \pa_3^{-1} \curl \left( \curl B \times b \right).  $$
It is easily seen that $b \mapsto \curl \left( u_{m,1}[b] \times B \right)$ is skew-adjoint with respect to the $L^2$ scalar product. Moreover, relying on the identity \eqref{curl_identity}, we can write:
\begin{equation}
\begin{aligned}
 \curl \left( u_{m,1}[b] \times B \right) 
& = \: B_{\rm h} \cdot\na_{\rm h}   \pa_3^{-1}  \left(B_{\rm h} \cdot\na_{\rm h} \curl  b   -   \curl b  \cdot \na B \right) \\ 
& \quad -   \pa_3^{-1} \left(  B_{\rm h} \cdot\na_{\rm h} \curl  b   - \curl b  \cdot \na B \right)   \cdot \na B \\
& =  \: B_{\rm h} \cdot\na_{\rm h} \pa_3^{-1} B_{\rm h} \cdot\na_{\rm h} \left( \begin{smallmatrix}\na_{\rm h} \\ 0 \end{smallmatrix} \right)   \times b  \: + \:  \mathcal{R}_{m,1} b  \, .
\end{aligned}
\end{equation}
Again,  $b \mapsto B_{\rm h} \cdot\na_{\rm h} \pa_3^{-1} B_{\rm h} \cdot\na_{\rm h} \left( \begin{smallmatrix}\na_{\rm h} \\ 0 \end{smallmatrix} \right)   \times b$ is skew-adjoint with respect to the $L^2$ scalar product. It is a third operator in variables $x_1, x_2$, regularizing in $x_3$ thanks to   $\pa_3^{-1}$.  The remaining operator $\mathcal{R}_{m,1}$ is also skew-adjoint (because the total operator is), second order in $x_1,x_2$, bounded in $x_3$.   
As regards $b \rightarrow \curl(u_{m,2}[b] \times B)$, we use again identity \eqref{curl_identity}, and write 
\begin{equation}
\begin{aligned}
  \curl \left( u_{m,2}[b] \times B \right)  & =  B_{\rm h} \cdot\na_{\rm h}  \pa_3^{-1}  \left( b \cdot \na \curl B - (\curl B)_{\rm h} \cdot\na_{\rm h} b \right) \\     
&  -   \pa_3^{-1}  \left( b \cdot \na \curl B - (\curl B)_{\rm h} \cdot\na_{\rm h} b \right) \cdot \na B  \, .
\end{aligned}
\end{equation}	
We recall that $B$ depends only on $x_3$, so that  the vertical component of $\curl B$ is zero. 
In the right-hand side, only the term  $b \rightarrow - B_{\rm h} \cdot\na_{\rm h}  \pa_3^{-1} (\curl B)_{\rm h} \cdot\na_{\rm h} b$ is second order in $x_1,x_2$ (and regularizing in $x_3$), all other terms are first order in $x_1,x_2$ (and regularizing  in $x_3$).  We can split this second order term into self-adjoint and skew-adjoint terms:  we end up with
\begin{equation}
    \curl \left( u_{m,2}[b] \times B \right)  
      = - \frac 12 \big( B_{\rm h} \cdot\na_{\rm h}  \pa_3^{-1}  (\curl B)_{\rm h} \cdot\na_{\rm h} b - (\curl B)_{\rm h} \cdot\na_{\rm h} \pa_3^{-1} B_{\rm h} \cdot\na_{\rm h} b \big) +  \mathcal{R}_{m,2} b 
\end{equation} 
where the operator $\mathcal{R}_{m,2}$ gathers a second order skew-adjoint operator and a first order operator in $x_1, x_2$. 
Eventually, we can write 
\begin{equation}  \label{decompo_um}
\curl \left( u_m \times b \right)  \: = \: \mathcal{A}_m b \: + \: \mathcal{C}_m b \: + \: \mathcal{R}_m b 
\end{equation}
where~$ \mathcal{A}_m $ and~$ \mathcal{C}_m $ are respectively    third and second order operators, 
with~$ \mathcal{A}_m $
 skew-adjoint  and~$ \mathcal{C}_m $ self-adjoint, defined by
\begin{equation} \label{defiAmCm}
\begin{aligned}
 \mathcal{A}_m b \: & := \: B_{\rm h} \cdot\na_{\rm h} \pa_3^{-1} B_{\rm h} \cdot\na_{\rm h} \left( \begin{smallmatrix}\na_{\rm h} \\ 0 \end{smallmatrix} \right)   \times b \\
\mathcal{C}_m b \: &  := \: - \frac{1}{2}B_{\rm h} \cdot\na_{\rm h}  \pa_3^{-1}  (\curl B)_{\rm h} \cdot\na_{\rm h} b \: + \: \frac{1}{2}(\curl B)_{\rm h} \cdot\na_{\rm h} \pa_3^{-1} B_{\rm h} \cdot\na_{\rm h} b
\end{aligned}
\end{equation}
and where $\mathcal{R}_{m}$ is the sum of a skew-adjoint operator of second order in $x_1,x_2$ and of a first order operator in $x_1,x_2$ (both bounded in $x_3$). In particular
\begin{equation} \label{estimRm}
\begin{aligned}
 \| \mathcal{R}_m b \|_{L^2(\T^3)} \:   \le \: C\, \| b \|_{H^2(\T^3)} \, , & \quad  \Big| \int_{\T^3}  \mathcal{R}_m b \cdot  b  \, dx\Big| \: \le \: C \, \| b \|_{H^1(\T^3)}   \,  \| b \|_{L^2(\T^3)} \, ,  \\ 
 & \text{for all smooth fields $b$ with zero mean.}
\end{aligned}
\end{equation}

\subsection{Computation of the geostrophic field}\label{computationgeostropic}
Like in the nonlinear Taylor system, any solution $u$ of (\ref{linear}a) reads
$u = u_m + u_g$, where $u_m$ is defined in \eqref{formulaum}, and $u_g$ is in the kernel of the Coriolis operator. As before, $u_g = u_g(t,x_{\rm h})$, with $\Div_{\rm h} u_{g,{\rm h}}  = 0$. Then, 
\begin{equation} \label{induction_linear}
 \pa_t b = \curl(u_m \times B) + \curl(u_g \times B) + \Delta b \, . 
 \end{equation}
 The term $\curl(u_g \times B)$ is there to preserve  the linear Taylor constraint \eqref{linear_constraint} through time. To determine $u_g$, we proceed exactly as in the nonlinear case, see the lines around \eqref{Stokes_ug}  and~ \eqref{Poisson_ug}. Here, we get 
\begin{equation} \label{linear_ug}
 \left( \begin{smallmatrix} \mathbb{P}_{\rm h} &  \\ & 1 \end{smallmatrix}  \right) (-\Delta_B) u_g =    \left( \begin{smallmatrix} \mathbb{P}_{\rm h} &  \\ & 1 \end{smallmatrix}  \right)  \int_\T F \, dx_3 
 \end{equation} 
 where 
 $$ F = B_{\rm h} \cdot\na_{\rm h} (\curl(u_m \times B) + \Delta b) +  (\curl(u_m \times B) + \Delta b) \cdot \na B  $$ 
 and 
 $$ \Delta_B \: := \: \int_\T (B_{\rm h} \cdot\na_{\rm h})^2 \, dx_3 \, . $$
The proof of Proposition \ref{solvability_ug}, on the well-posedness of \eqref{Stokes_ug}-\eqref{Poisson_ug},  applies to \eqref{linear_ug}. Under assumption \eqref{nonconstant}, if $b$ is smooth  (which implies that  $u_m$ is smooth),  \eqref{linear_ug} has a unique smooth  solution with zero mean: 
$$u_{g}  = \left( \begin{smallmatrix} \mathbb{P}_{\rm h} &  \\ & 1 \end{smallmatrix}  \right) (-\Delta_B)^{-1}  \int_\T F \, dx_3. $$
Note that $\mathbb{P}_{\rm h}$ and $\Delta_B$ commute, since~$B$ only depends on~$x_3$. Note also that $\Delta_B$ is a second order operator,  uniformly elliptic by assumption \eqref{nonconstant}. Hence, $(-\Delta_B)^{-1}$ gains two derivatives in the Sobolev scale. 
Taking this into account, and using the decomposition \eqref{decompo_um}, we  can write 
\begin{equation} 
u_g  \: = \: u_g^{high}[b] \: + \: u_{g}^{low}[b] 
\end{equation}
where~$u_g^{high}$ and~$u_g^{low}$ are both quasigeostrophic and satisfy 
\begin{equation}\label{defughigh}
u_g^{high}[b]   =  \left( \begin{smallmatrix} \mathbb{P}_{\rm h} &  \\ & 1 \end{smallmatrix}  \right) (-\Delta_B)^{-1} \int_\T B_{\rm h}(x_3) \cdot\na_{\rm h} \,  \mathcal{A}_m b(\cdot,x_3) \, dx_3 \, ,
\end{equation}
and where $u_g^{low}$ satisfies the estimate 
\begin{equation}
 \| u_g^{low}[b] \|_{H^s(\T^3)} \: \le \: C \,  \| b \|_{H^{s+1}(\T^3)} 
 \end{equation} 
for all~$s \geq 0$ and all smooth divergence-free fields $b$ with zero mean. It follows that 
\begin{equation} \label{uglow_1}
\Big \| \curl \Big( u_g^{low}[b] \times B \Big)\Big  \|_{H^s(\T^3)} \: \le \: C \,  \| b \|_{H^{s+2}(\T^3)}  \, .
\end{equation}
The operator~$ \curl  ( u_g^{low}[\cdot] \times B  )$ therefore is a second order operator, but one notes that if $b$ satisfies \eqref{linear_constraint}, one has 
$$  \int_{\T^3} \left(\curl b \times B  + \curl B \times b \right)  \cdot v_g  \,  \: = 0  \quad \text{ for all geostrophic fields $v_g$} \, . $$
With the choice $v_g = u_g^{low}$, we therefore obtain that 
 \begin{equation*}
 \begin{aligned}
\int_{\T^3}  \curl (u_g^{low}  \times B) \cdot  b  & =  \int_{\T^3}  u_g^{low} \cdot   (B    \times \curl b)   \\
& =  \int_{\T^3}  u_g^{low} \cdot \left(  \curl B \times b \right)
 \end{aligned}
\end{equation*} 
so that 
\begin{equation}  \label{uglow_2}
\left| \int_{\T^3}  \curl (u_g^{low}  \times B) \cdot  b  \right| \: \le \: C \| b \|_{H^1(\T^3)} \, \| b \| _{L^2(\T^3)} \, .
\end{equation}
Hence, the operator $ \curl  ( u_g^{low}[\cdot] \times B  )$ will act as a first order operator in an~$L^2$ energy estimate.
Eventually, as regards~$u_g^{high}[b]$, we can notice that 
$$ \int_\T (B_{\rm h} \cdot\na_{\rm h} \mathcal{A}_m b)_{\rm h}  dx_3  =-\na_{\rm h}^{\perp} \int_\T \left(  (B_{\rm h} \cdot\na_{\rm h})^2  \pa_3^{-1} (B_{\rm h} \cdot\na_{\rm h}) b \right)_3 \, dx_3 $$
where we have noted~$\na_{\rm h}^{\perp}  =\left(\begin{smallmatrix} -\partial_2 \\  \partial_1 \\ 0 \end{smallmatrix}  \right)$ so that it is a 2d divergence-free field. Thus, there is no need for $\mathbb{P}_{\rm h}$ in~(\ref{defughigh}), and we find 
\begin{equation} 
u_g^{high}[b] = (-\Delta_B)^{-1} \int_\T B_{\rm h}(x_3) \cdot\na_{\rm h} \,  \mathcal{A}_m b(\cdot,x_3) \, dx_3 \, . 
\end{equation}
We can then write 
$$ \curl(u_g^{high} \times B) = B_{\rm h} \cdot\na_{\rm h} u_g^{high} - u_g^{high} \cdot \na B = \Pi \mathcal{A}_m b - u_g^{high} \cdot \na B  $$
where $\Pi$ is the self-adjoint bounded operator defined by 
\begin{equation}
 \Pi  f \:: = \:   B_{\rm h} \cdot \na_{\rm h}  (-\Delta_B)^{-1}  \int_{\T}  B_{\rm h} \cdot\na_{\rm h} f \, dx_3 \, . 
\end{equation} 
Eventually, to mimic decomposition \eqref{decompo_um}, we write 
\begin{equation}
\curl(u_g \times B) = \mathcal{A}_g b + \mathcal{C}_g b + \mathcal{R}_g b  
\end{equation}
where 
\begin{equation}  \label{defiAgCgRg}
\begin{aligned} 
\mathcal{A}_g b  \:  & := \:  \Pi \mathcal{A}_m b  \: +\mathcal{A}_m \Pi b \, , \\  
\mathcal{C}_g b  \:  & := \:  - \mathcal{A}_m \Pi b - u_g^{high}[b]  \cdot \na B \, , \\
\mathcal{R}_g b  \:  & :=  \: \curl (u_g^{low}[b]  \times B) \, .
\end{aligned}
\end{equation} 
Note that $\mathcal{A}_g$, like~$\mathcal{A}_m$,  is skew-adjoint. In contrast~$\mathcal{C}_g$  is not self-adjoint, contrary to~$\mathcal{C}_m$. Also, by \eqref{uglow_1}-\eqref{uglow_2} we have for all  smooth $b$ with zero mean
\begin{equation} \label{estimRg}
\begin{aligned} 
  \| \mathcal{R}_g b \|_{L^2(\T^3)} & \le  C\, \| b \|_{H^2(\T^3)}\quad \text{and}   \\
  \Big| \int_{\T^3}  \mathcal{R}_g b \cdot  b \, dx \Big| &\le  C \, \| b \|_{H^1(\T^3)}   \,  \| b \|_{L^2(\T^3)}  \quad  \text{for all smooth $b$  satisfying \eqref{linear_constraint} .}
\end{aligned}
\end{equation}

\subsection{Fourier transform}
From the two previous paragraphs, we can reformulate the linearized Taylor system \eqref{linear}  in terms of $b$ only.  The evolution equation on $b$ is
\begin{equation} \label{induction_new}
\pa_t b  = \mathcal{A} b + \mathcal{C} b + \mathcal{R} b   + \Delta b
\end{equation} 
with 
$$ \mathcal{A} = \mathcal{A}_m + \mathcal{A}_g \, , \quad \mathcal{C} = \mathcal{C}_m + \mathcal{C}_g \, , \quad \mathcal{R} = \mathcal{R}_m + \mathcal{R}_g \, , $$
see \eqref{defiAmCm}-\eqref{estimRm},\eqref{defiAgCgRg}-\eqref{estimRg}. Moreover, the solution $b$ should satisfy the linear Taylor constraint and divergence-free constraints, namely
\begin{equation}  \label{constraints}
 \left( \begin{smallmatrix} \mathbb{P}_{\rm h} &  \\ & 1 \end{smallmatrix}  \right)   \int_{\T^3} \left(\curl b \times B  + \curl B \times b \right) = 0 \, , \quad \Div b =0 \, .
 \end{equation}

\medskip
The  point is to establish the {\it a priori} estimate \eqref{theresult} for smooth solutions $b$ of \eqref{induction_new}-\eqref{constraints}.
 This is easier after taking a horizontal Fourier transform, since~$B$ does not depend on~$x_{\rm h}  $.  Given horizontal Fourier modes~$\xi:=(\xi_1,\xi_2) \in (2\pi \Z)^2$, we write
$$
b(t,x) = e^{i\xi\cdot x_{\rm h}} \widehat b(t,\xi,x_3) \, .   
$$
Is is easily seen that the zero  mode  $\widehat b(t,0,x_3)$ satisfies the heat equation, so that it decays exponentially. We can therefore focus on the case $\xi \neq 0$.  From now on we omit the dependence of $\widehat b$ on~$\xi$ which is fixed and simply write~$ \widehat b(t,x_3)$.

 We  introduce the notation, recalling that~$\xi^{\perp}:  =\left(\begin{smallmatrix} -\xi_2 \\  \xi_1 \end{smallmatrix}  \right)$,
\begin{equation}  \label{defi_beta}
\beta_\xi(x_3) \: := \: B_{\rm h}(x_3) \cdot \xi\,,  \quad \beta'_{\xi^\perp}(x_3) \: := \:      B_{\rm h}'(x_3) \cdot \xi^\perp\,, \quad e_0(x_3) := \frac{\beta_\xi(x_3)}{\| \beta_\xi \|_{L^2(\T)}} \, \cdotp 
\end{equation}
Let us stress that: 
$$ \exists c, \, C > 0, \quad         c \, |\xi| \le  \| \beta_\xi \|_{L^2(\T)} \: \le \:  C \, |\xi|, \quad \forall \xi \in (2\pi\Z)^2 \, ,$$
the lower bound coming from \eqref{nonconstant}. 
We denote by $\Pi_{e_0}$ the orthogonal projection over $\R e_0$ in $L^2(\T)$. Namely
$$ \Pi_{e_0} f \: := \:  \left( \int_\T e_0  \,  f \, dx_3  \right) \, e_0  \: =  \:   \frac{1}{\|\beta_\xi\|_{L^2(\T)}^2}  \left( \int_\T \beta_\xi  \, f \, dx_3 \right) \beta_\xi \, . 
$$
To lighten notation, we will also call $\Pi_{e_0}$  the orthogonal  projection over $\R^3 e_0$ in $L^2(\T)^3$, acting component-wise: 
$\widehat b \mapsto (\Pi_{e_0} \widehat b_1, \Pi_{e_0} \widehat b_2, \Pi_{e_0} \widehat b_3)$.  

With such notations,  we find that~$\widehat \Pi = - \Pi_{e_0}$ and therefore
$$\pa_t \widehat b = \widehat{\mathcal{A}} \  \widehat b \: + \: \widehat{\mathcal{C}}  \  \widehat b +  \widehat{\mathcal{R}}  \,  \widehat b \: + \: 
(\pa^2_3 - |\xi|^2 )\widehat b \, , $$
with
\begin{equation}
\widehat{\mathcal{A}} \ \widehat b \: := \:  \left( \widehat{\mathcal{A}_m} - \Pi_{e_0}  \widehat{\mathcal{A}_m}  -   \widehat{\mathcal{A}_m} \Pi_{e_0} \right) \widehat b \, , 
 \end{equation}
\begin{equation} \label{defiChat}
\widehat{\mathcal{C}} \ \widehat b \:  := - \frac{1}{2} \beta_\xi \pa_3^{-1} \beta'_{\xi^\perp}  \, \widehat b +\frac{1}{2}  \beta'_{\xi^\perp} \pa_3^{-1} \beta_\xi \, \widehat b 
\: +  \:  \widehat{\mathcal{A}_m} \Pi_{e_0}  \widehat b  \:  - \frac{i}{\| \beta_\xi \|_{L^2(\T)}^2} \int_\T \beta_\xi  \left( \widehat{\mathcal{A}_m} \, \widehat b \right)_3 \, dx_3 \, B' \, ,
\end{equation}
where 
\begin{equation} 
\widehat{\mathcal{A}_m} \,  \widehat b  \: := \: -i \beta_\xi \pa_3^{-1} \beta_\xi \left( \begin{smallmatrix} \xi \\ 0 \end{smallmatrix} \right) \times \widehat b \, ,
\end{equation}
and with, for all smooth  $b$ with zero mean
\begin{equation}  \label{estimRhat}
\begin{aligned}
\| \widehat{\mathcal{R}} \, \hat{b} \|_{L^2(\T)}& \le C \, |\xi|^2 \, \| \widehat b \|_{L^2(\T)} \quad \text{and} \\
 \big|  \big( \widehat{\mathcal{R}} \, \widehat{b} |  \widehat{b} \big) \big|& \le   C  \,   |\xi|   \, \| \widehat b \|_{L^2(\T)}  \quad  \text{for all smooth $b$ satisfying \eqref{linear_constraint}.}
 \end{aligned}
\end{equation}
See \eqref{estimRm}-\eqref{estimRg}. Here, $( \: | \: )$ is the usual scalar product on $L^2(\T,\bC)$.

The operator $\widehat{\mathcal{A}}$ is skew-selfadjoint, and therefore satisfies, 
\begin{equation} \label{estimAhat}
 ( \widehat{\mathcal{A}} \ \widehat{b} | \widehat{b} )   \: = \: 0\,.
 \end{equation}
Furthermore, one can notice that  
\begin{equation} \label{PiAmPi}
\Pi_{e_0} \widehat{\mathcal{A}_m} \Pi_{e_0} = 0
\end{equation}
 due to the fact that for all $V \in \R^3$: 
$$ \int_\T \beta_\xi \cdot \widehat{\mathcal{A}_m} \, (\beta_\xi V)  \, dx_3 \: = \: -i \int_\T\beta_\xi^2 \pa^{-1}_3 \beta_\xi^2 \, dx_3 \,  \left( \begin{smallmatrix} \xi \\ 0 \end{smallmatrix} \right) \times V  = 0\,. $$
Hence, we can write in a condensed way: 
\begin{equation} 
\widehat{\mathcal{A}}  = \Pi_{e_0}^\perp \widehat{\mathcal{A}_m}  \Pi_{e_0}^\perp \, , \quad \Pi_{e_0}^\perp = \rm{Id}  -   \Pi_{e_0}\,. 
\end{equation}
As regards the operator $\widehat{\mathcal{C}}$, we find 
\begin{equation} \label{estimChat_1}
 \| \widehat{\mathcal{C}} \  \widehat{b} \|_{L^2(\T^3)} \: \le \: C  \left( | \xi |^2  \| \widehat{b} \|_{L^2(\T^3)} +  | \xi |^3  \| \Pi_{e_0} \widehat{b} \|_{L^2(\T^3)} \right) 
 \end{equation} 
where the $O(|\xi|^3)$ comes from the term  $\widehat{\mathcal{A}_m} \Pi_{e_0}$ in the expression for  $\widehat{\mathcal{C}}$. 
We now state the following useful lemma.
\begin{lem}\label{taylorpie0}
Let~$b$ satisfy the Taylor constraint~{\rm(\ref{linear_constraint})}. Then
\begin{equation} \label{formulaPi0b}
\Pi_{e_0}\widehat  b = \frac{\beta_\xi}{\| \beta_\xi \|_{L^2(\T)}^2} \int_\T \left( i \, \widehat b_3  \, B' \: - \:  2i |\xi|^{-2}  \, \beta_\xi ' \, \widehat b_3\,  \left(\begin{smallmatrix}  \xi  \\  0  \end{smallmatrix} \right) \,  \right) \, dx_3 \, .
\end{equation}
In particular
$$
\big\| \Pi_{e_0}\widehat  b \big \|_{L^2} \le \frac{C}{|\xi|}  \|\widehat  b \|_{L^2} \, ,
$$
while 
\begin{equation} \label{AmPie0}
\widehat{\mathcal{A}_m} \Pi_{e_0}\widehat  b =  \widehat{\mathcal{A}_m} \left( \frac{\beta_\xi}{\| \beta_\xi \|_{L^2(\T)}^2} \int_\T  i \, \widehat b_3  \, B' 
dx_3\right)\,.
\end{equation}
\end{lem}   
\begin{proof}  We notice that
the Fourier transform of the Taylor constraint~(\ref{linear_constraint}) can be written
$$
 \int_\T\big(i  \beta_\xi \,   \widehat b  +B' \, \widehat b_3  \big) \, dx_3 = i \left(\begin{smallmatrix}  \xi  \\  0  \end{smallmatrix} \right) \, 
\widehat p \,.	$$
This implies on the one hand that
$$
\Pi_{e_0} \widehat b = \frac{    \beta_\xi}{ \|   \beta_\xi\|_{L^2(\T)}^2} \int_\T \big( i\widehat b_3 B'  + \left(\begin{smallmatrix}  \xi  \\  0  \end{smallmatrix} \right) \, \widehat p \big)\, dx_3
$$
and also
$$
\widehat p = \frac1{ |\xi|^2} \int_\T \big(\beta_\xi \,  \widehat b_{\rm h} \cdot \xi	 - i \widehat b_3 B'_{\rm h } \cdot \xi \big)\, dx_3 \, .
$$
The formula follows from the divergence free condition~$ \widehat b_{\rm h} \cdot \xi = i    \widehat b'_3$ and an integration by parts. 
\end{proof}
It follows from this lemma and  \eqref{estimChat_1} that 
\begin{equation} \label{estimChat}
 \| \widehat{\mathcal{C}} \  \widehat{b} \|_{L^2(\T^3)} \: \le \: C  \,  |\xi |^2  \| \widehat{b} \|_{L^2(\T^3)}\,.  
 \end{equation} 

 	\subsection{Reduction to large horizontal frequencies}\label{reductionlargehorizontal}
Let us write an~$L^2$ energy estimate on the equation satisfied by~$\widehat b$. From \eqref{estimRhat},\eqref{estimAhat} and \eqref{estimChat}, we deduce 
\begin{equation}\label{energyestimate}
\begin{aligned}
\frac12 \frac d {dt}\| \widehat{b}(t)\|_{L^2(\T)}^2  + |\xi|^2 \|\widehat{b}(t)\|_{L^2(\T)}^2  + \|\partial_3 \widehat{b}(t)\|_{L^2(\T)}^2  \leq C (1+|\xi|^2 )\|\widehat{b}(t)\|_{L^2(\T)}^2\, .
		  \end{aligned}
	\end{equation}
 	Gronwall's lemma gives therefore directly that
	$$
	\|\widehat{b}(t)\|_{L^2(\T)} \leq \| \widehat{b}(0) \|_{L^2(\T)} \, \exp\big (C (1+|\xi|^2) t\big)
	$$
	so from now on we may restrict our attention to the case when~$|\xi| \gg1$: we introduce
	$$\eps \:  :=  \: \frac{1}{|\xi|} \ll 1,   \quad \:   \eta \: :=  \: \eps \, \xi \, \in \mathbb{S}^1$$
and     express equation \eqref{induction_new} in terms of $\eps$ and $\eta$. Similarly to 
$\beta_\xi$ and $\beta'_{\xi^\perp}$, we define 
\begin{equation}  \label{defi_beta_eta}
\beta_\eta(x_3) \: := \: B_{\rm h}(x_3) \cdot \eta \, ,  \quad \beta'_{\eta^\perp}(x_3) \: := \:      B_{\rm h}'(x_3) \cdot \eta^\perp\,.  
\end{equation}
Note that  
\begin{equation}  \label{nonconstant2}
\delta:= \min_{\etah \in {\mathbb S}^1} \| \beta_\eta \|_{L^2(\T)}  \, > 0\,. 
\end{equation}
Instead of the operators 
$$ \widehat{\mathcal{A}} = \widehat{\mathcal{A}}(\xi,\pa_3)\,,  \quad  \widehat{\mathcal{C}} = \widehat{\mathcal{C}}(\xi,\pa_3),  \quad \widehat{\mathcal{R}} = \widehat{\mathcal{R}}(\xi,\pa_3)\,,$$
we introduce 
$$ A = A(\eta,\pa_3) \: := \:  \eps^3  \widehat{A}\left(\frac{\eta}{\eps}, \pa_3\right) \, , \quad C = C(\eta,\pa_3) \: := \: \eps^2 \,   
\widehat{C}\left(\frac{\eta}{\eps}, \pa_3\right)\,,  $$
and 
$$ R \: = \: \,   R(\eta,\pa_3) \: := \:   \eps^2 \widehat{R}\left(\frac{\eta}{\eps}, \pa_3\right)\,. $$
We also drop the $ \:\: \widehat{} \:\: $ on $\widehat{b}$, denoting $b$ instead. 
We have notably
\begin{equation} \label{defiA}
A \, b \: = \:   \Pi_{e_0}^\perp  \, A_m \,  \Pi_{e_0}^\perp  b\,,  
 \end{equation}
 and
\begin{equation} \label{defiC}
C \, b \:   :=  C_m  b \: + \: A_m  \left( \frac{\beta_\eta}{\| \beta_\eta \|_{L^2(\T)}^2} \int_\T  i \,  b_3  \, B'  \, dx_3 \right)   \:  - \:  \frac{i}{\| \beta_\eta \|_{L^2(\T)}^2} \int_\T \beta_\eta  \left( A_m \,  b \right)_3 \, dx_3 \, B'\end{equation}
where 
\begin{equation} \label{defAmCm}
A_m \,  b  \: := \: -i \beta_\eta \pa_3^{-1} \beta_\eta \left( \begin{smallmatrix} \eta \\ 0 \end{smallmatrix} \right) \times b, \quad C_m b  \: = \:  -\frac{1}{2} \beta_\eta \pa_3^{-1} \beta'_{\eta^\perp}  \,  b +  \frac{1}{2}  \beta'_{\eta^\perp} \pa_3^{-1} \beta_\eta \, \ b\,. 
\end{equation}
Note that the second term at the right-hand side of \eqref{defiC} comes from \eqref{AmPie0}. 
We get
\begin{equation} \label{induction3}
\pa_t b = \frac{A b}{\eps^3} + \frac{C b}{\eps^2} + \frac{Rb}{\eps^2} - \frac{1}{\eps^2} b + \pa_3^2 b  \, .
\end{equation}
Operators $A$ and  $C$  are independent of $\eps$. Moreover, $(A b |  b) = 0$, where $( \: | \: )$ is the usual scalar product over $L^2(\T ; \bC)$. 
 The remainder  $R$ is bounded uniformly in $\eps$, and   \eqref{estimRhat} implies that 
\begin{equation} \label{Rbb}
\frac{1}{\eps^2} \left| ( R b | b ) \right| \: \le \: \frac{C}{\eps} \| b \|_{L^2(\T^3)}^2
\end{equation}
for all smooth $b$ satisfying the Fourier version of the Taylor constraint
$$  \left( \begin{matrix}  \rm{Id} - |\xi|^{-2} \xi \otimes \xi & \\ & 1 \end{matrix}   \right)  \int_\T \left( \beta_\eta b + \eps b_3 B' \right) \, dx_3  = 0\,.$$
So, for $\eps$ small enough, it can be controlled by the term $-\frac{1}{\eps} \| b \|_{L^2(\T^3)}^2$  coming from the diffusion. 
The obstacle to estimate \eqref{theresult} is therefore the term $ \frac{1}{\eps^2}( C b | b )$.  

\subsection{General strategy} \label{sec_strategy}
To prove Theorem \ref{mainthm}, we shall resort to a normal form argument. In this section we present the method in a formal way. We denote generically by~$O(\eps^\alpha)$ an operator which may depend on on~$\eps$ and~$\eta$, but  is uniformly bounded by~$\eps^\alpha$ in operator norm over ~$L^2(\T)$.   
The idea is to change unknown by defining
		$$
		d:=({\rm{Id} }+ \eps Q) b
		$$
		with~$Q  = O(1)$. We then expect that  
		$$
		b = ({\rm{Id} }+ \eps Q) ^{-1}d = ({\rm{Id} }- \eps Q) d + O(\eps^2 )d 
		$$
				and
				$$
				\partial_t d = ({\rm{Id} }+ \eps Q)\partial_t b +O(\eps )d \, .
				$$
It follows that~$d$ should satisfy an equation of the type
				$$
\partial_t d -(\partial_3^2 - \frac1{\varepsilon^2})d  = \frac1{\varepsilon^3}  A d +  \frac1{\varepsilon^2} \big(   C  + [Q,A] + R\big)d  + O \Big(  \frac1{\varepsilon} \Big)d\, .				
				$$
The idea is to take  $Q$ as a solution of the homological equation 
$[A,Q]  =  C$. We refer to \cite{CGM,BGG,GSP} for applications of this strategy. 
Nevertheless, solving this equation is difficult in our case. To explain why, let us consider a simplified version where we neglect all terms coming from the geostrophic part $\curl(u_g \times b)$. This means we consider the equation 
$$ [A_m, Q]  = C_m $$
with $A_m$ and $C_m$ defined in \eqref{defAmCm}.

If the matrix $ \: \left( \begin{smallmatrix} \eta \\ 0 \end{smallmatrix} \right) \times$ were invertible, and if the function $\beta_\eta$ were not vanishing anywhere on $\T$, then a natural candidate for $Q$ would be: 
$$ Q  \: = \: -\frac{i}{2} \frac{\beta'_{\eta^\perp}}{\beta_\eta} \left(\left( \begin{smallmatrix} \eta \\ 0 \end{smallmatrix} \right) \times\right)^{-1} $$  
which  satisfies formally $[A_m, Q]  = C_m$.  

Unfortunately, trying to make this kind of construction rigorous, we   face several difficulties: 
\begin{enumerate}
\item The matrix  $\left( \begin{smallmatrix} \eta \\ 0 \end{smallmatrix} \right) \times$ is not invertible. 
\item The full expression of $A$ and $C$ involves additional terms, related to the geostrophic field $u_g$, notably the orthogonal projection $\Pi_{e_0}$. 
\item For any $x_3$, there is some $\eta \in \mathbb{S}^1$ such that $\beta_\eta(x_3) = 0$. From a different perspective, one may say that the multiplication by $\frac{1}{\beta_\eta}$ is not a bounded operator over $L^2(\T)$. 
\end{enumerate}
Hence, the normal form  cannot be  applied directly, and we   need additional arguments to overcome the issues just mentioned. 
 Roughly, the first difficulty, related to the kernel of~$\left( \begin{smallmatrix} \eta \\ 0 \end{smallmatrix} \right) \times$ will be handled thanks to the divergence-free constraint, which makes $b(t,x_3)$ almost orthogonal to this kernel (up to a power of $\eps$). The second one  will be handled taking advantage of  the Taylor constraint, notably through the identity of  Lemma \ref{taylorpie0}, which now reads 
 \begin{equation} \label{formulaPi0b_bis}
\Pi_{e_0} b = \eps \frac{\beta_\eta}{\| \beta_\eta \|_{L^2(\T)}^2} \int_\T \left( i \, \widehat b_3  \, B' \: - \:  2i   \, \beta_\eta ' \, b_3\,  \left(\begin{smallmatrix}  \eta  \\  0  \end{smallmatrix} \right) \,  \right) \, dx_3 \, .
\end{equation}

   Eventually, the last problem will be overcome by a spectral truncation. This means we shall only perform the construction of $Q$ in a finite-dimensional setting, projecting on the low modes of the skew-self-adjoint operator $A$. The high modes (that as we will show correspond to high frequencies in $x_3$), will be controlled, and discarded,  
 thanks to the presence of the operator~$\pa_3^{-1}$  in $C$. 

\section{Normal form argument}

\subsection{Using the divergence-free constraint}   \label{A2}
	Let $\Pi_\eta$ the orthogonal projection in $\bC^3$ over the vector $ ( \eta, 0 )^{\rm t}$, and $\Pi_\eta^\perp := \rm{Id} - \Pi_\eta$. 
	We shall use the divergence free constraint on $b$, which now reads 
\begin{equation} \label{divfreexih}
i \eta \cdot b_{\rm h} + \eps \pa_3 b_3 = 0 
\end{equation}
to show that we can somehow restrict to the control of $\Pi_\eta^\perp b$. More precisely we have the following result.
\begin{Prop} \label{estimateC2}
	For all $b,c$ in $L^2$, divergence free in the sense of \eqref{divfreexih}, we have 
	$$ \left| (C b | c ) -  (C \Pi_\eta^\perp b | \Pi_\eta^\perp c )\right|  \le C \,  \eps \| b  \|_{L^2(\T)} \,  \| c  \|_{L^2(\T)}  \, . $$
	\end{Prop}
	\begin{proof}
	Let us start by proving that
\begin{equation}\label{C2Pi+small}
 \|C \Pi_\eta  b\|_{L^2(\T)}  \leq C \eps  \| b  \|_{L^2(\T)}  \, . 
 \end{equation}
	From \eqref{divfreexih}, 
	$$
	\Pi_\eta b   = -(\eps \pa_3 (\Pi_\eta^\perp b)_3) \left( \begin{matrix} \eta \\ 0 \end{matrix} \right)\, .	$$
From \eqref{defiC}, we get  
\begin{equation} 
\begin{aligned}
C  \Pi_\eta b   \:   & :=  \: - \frac{1}{2} \beta_\eta \pa_3^{-1} \beta'_{\eta^\perp}  \,  \Pi_\eta b +  \frac{1}{2}  \beta'_{\eta^\perp} \pa_3^{-1} \beta_\eta \, \  \Pi_\eta b \\
& =  \eps \left( \frac{1}{2} \beta_\eta \pa_3^{-1} \left( \beta'_{\eta^\perp}  \, \pa_3 (\Pi_\eta^\perp b)_3 \right) -  \frac{1}{2}  \beta'_{\eta^\perp} \pa_3^{-1} \left( \beta_\eta  \, \pa_3 (\Pi_\eta^\perp b)_3 \right) \right) \left( \begin{matrix} \eta \\ 0 \end{matrix} \right)
\end{aligned}
\end{equation}
We notice that 
$$ \pa_3^{-1} (\beta_\eta \pa_3 \Pi_\eta^\perp b) = \beta_\eta \Pi_\eta^\perp b   - \int_\T (\beta_\eta \Pi_\eta^\perp b) -  \pa_3^{-1} \left(  \beta'_\eta \Pi_\eta^\perp b \right)  $$
so that
\begin{equation} \label{bounded}
 \|  \pa_3^{-1} (\beta_\eta \pa_3 \Pi_\eta^\perp b) \|_{L^2(\T)} \le C \| b \|_{L^2(\T)}  
 \end{equation}
and the same with $\beta'_{\eta^\perp}$ instead of $\beta_\eta$. Inequality~(\ref{C2Pi+small}) follows.
	Then to end the proof, writing
	$$
	 (C b | c ) =  (C \Pi_\eta b |  c ) + (C \Pi_\eta^\perp b | \Pi_\eta c )+ (C \Pi_\eta^\perp b | \Pi_\eta^\perp c ) \, , 
	$$
	it suffices to prove that
	$$
\big |	 (C \Pi_\eta^\perp b | \Pi_\eta c )\big |  \leq C \eps  \| b  \|_{L^2(\T)}\| c  \|_{L^2(\T)} \, .
	$$
We have 
	$$
	\begin{aligned}
 (C \Pi_\eta^\perp b | \Pi_\eta c ) &=  - \eps \left( C \Pi_\eta^\perp b \:  | \: \pa_3 (\Pi_\eta^\perp c)_3 \left( \begin{smallmatrix} \eta \\ 0 \end{smallmatrix} \right) \right)\\
 & = \eps \left( \pa_3 C \Pi_\eta^\perp b  \: | \: (\Pi_\eta^\perp c)_3 \left( \begin{smallmatrix} \eta \\ 0 \end{smallmatrix} \right) \right) \, .
 \end{aligned}
 $$
Using again \eqref{bounded}, it is easily seen that  $\| \pa_3 C \Pi_\eta^\perp b \|_{L^2(\T)} \le C \| \Pi_\eta^\perp b \|_{L^2(\T)}$. The result follows from the Cauchy-Schwarz inequality. 
	\end{proof}
 	\subsection{Spectral analysis of $A$}\label{spectralanalysis}
		In this paragraph we diagonalize the operator~$A$, 	recalling
			$$
		A b := i  \Pi_{e_0}^\perp 	 \beta_\eta \partial_3^{-1} 
	 \beta_\eta \Pi_{e_0}^\perp b  \times  	\left(\begin{array}{c}
		\eta \\ 0
		\end{array}\right)    \,  .
	 $$
	 		The following result holds. 
\begin{Prop}\label{spectraldec}
		Given $\eta \in {\mathbb S}^1$,  recall  
\begin{equation} \label{defe0}
e_0 \:  :=  \: \beta_\eta / \| \beta_\eta \|_{L^2(\T)} 
\end{equation}
and for all $k \in \Z^*$ define
\begin{equation} \label{defeklambdak} 
\begin{aligned} 
\mu_{k} \: & :=  \:  \frac{ \|\beta_\eta \|_{L^2(\T)}^2 }{2ik\pi} \\
e_k \: & := \: \beta_{\eta,k} /\|  \beta_{\eta,k} \|_{L^2(\T)} \, ,   \quad \mbox{with} \quad	 \beta_{\eta,k}(x_3):= \beta_\eta( x_3)  \exp \Big(\frac1{ \mu_k} \int_0^{x_3} \beta_\eta^2 ( y_3) \, dy_3\Big) \,.  \\
\end{aligned}
\end{equation}
Then, the set~$( \Phi_{k}^\pm)_{k \in \Z}$  defined  by 		
		$$
		\begin{aligned}
		 \Phi_{k}^-:=		\left(\begin{array}{c}
		\eta \\ 0
		\end{array}\right)   e_k \, ,  \quad
 \Phi_{2k}^+ :=\frac1{\sqrt2}\left(\begin{array}{c}
		\eta ^\perp\\ i
		\end{array}\right)  e_k \,  , \quad 	 \Phi_{2k+1}^+ :=\frac1{\sqrt2}\left(\begin{array}{c}
		\eta ^\perp\\- i
		\end{array}\right) e_{k}
		\end{aligned}
		$$
is an orthonormal basis of~$L^2(\T)^3$ satisfying  
$$ A \Phi_k^-  =  A \Phi^+_0 =  A \Phi^+_1 = 0 \, , \quad A \Phi^+_{2k} = \mu_k \Phi^+_{2k} \, , \quad  A \Phi^+_{2k+1} = -\mu_k \Phi^+_{2k+1}, \quad k \in \Z^* \, .$$
		\end{Prop}
		\begin{proof}  
		We start by considering the compact operator
		$$
		D:=\Pi_{e_0}^\perp  \beta_\eta \,  \partial_3^{-1} 
	\beta_\eta  \Pi_{e_0}^\perp 
		$$
		which  clearly sends $e_0$ to $0$ and $\Pi_{e_0}^\perp  L^2(\T)$ to itself. Let~$\mu \neq 0$ be an eigenvalue of~$D$ in~$\Pi_{e_0}^\perp L^2(\T)$, and~$f$ an associate eigenfunction. Then since~$\Pi_{e_0}^\perp  f = f$ this means that~$f$ must solve, in~$\Pi_{e_0}^\perp L^2(\T)$, the equation
\begin{equation}\label{eq1}  
 \beta_\eta \,  \partial_3^{-1} 
	\beta_\eta f = \mu f + \alpha \beta_\eta \, , \quad \mbox{for some} \quad \alpha \in {\mathbb C} \, .
\end{equation}
Moreover~$\Pi_{e_0}^\perp  f = f$ means that~$\beta_\eta f$ has zero average with respect to $x_3$, so that $u:= \partial_3^{-1} (\beta_\eta f)$ satisfies $\pa_3 u  = \beta_\eta f$. Hence, $u$ must satisfy 
		$$
	 \beta_\eta u = \frac\mu  {\beta_\eta} \partial_3 u	+ \alpha \beta_\eta	$$
	 hence
	 $$
	 u(x_3)=  \exp \Big(\frac1 \mu \int_0^{x_3} \beta_\eta^2 (y_3) \, dy_3\Big)
	 - \int_\T\exp \Big(\frac1 \mu \int_0^{x_3} \beta_\eta^2 (y_3) \, dy_3\Big) \, dx_3
	  \, ,
	 $$
	 and~$\alpha = \displaystyle  - \int_\T \exp \Big(\frac1 \mu \int_0^{x_3} \beta_\eta^2 (y_3) \, dy_3\Big) \, dx_3$.
	But~$u (0) = u(1)$
	   therefore
	   $$
	    \exp \Big(\frac1 \mu \int_\T \beta_\eta^2 (y_3) \, dy_3\Big) = 1
	   $$
	   which implies that~$\mu = \mu_k$ where for~$k \in  \Z^* ,$
	 $$
\mu_k:= \frac1{2ik\pi} \int_\T \beta_\eta^2 (x_3) \, dx_3 \, .
	 $$
Finally we have
	 $$
	 f(x_3) =\frac{ 2ik\pi}  {\|\beta_\eta \|_{L^2}^2} \beta_\eta(x_3) \exp \Big(\frac{ 2ik\pi}  {\|\beta_\eta\|_{L^2}^2} \int_0^{x_3} \beta_\eta^2 (y_3) \, dy_3\Big)
	 \, .
	 $$
	  It follows that the family~$(e_k)_{k \in \Z^*}$ defined in (\ref{defeklambdak}) is an orthonormal basis of~$\Pi_{e_0}^\perp L^2 (\T)$, while $(e_k)_{k \in \Z}$ is an orthonormal basis of $L^2(\T)$.  	  
	  Finally to recover an orthonormal basis of eigenfunctions of~$A$ in~$(L^2(\T))^3$, we   use the fact that
	 $$
	 \left(\begin{array}{c}
	\eta^\perp \\\pm i
		\end{array}\right) \, , 	 \left(\begin{array}{c}
	 \eta  \\0
		\end{array}\right) 
	 $$
	 is a basis of eigenvectors of the operator~$i \cdot \times\left(\begin{array}{c}
	 \eta  \\0
		\end{array}\right)     $, where the two first are associated with the eigenvalues~$\pm 1$ and the third vector in its kernel. The result follows directly. 
	 \end{proof}
\noindent	 Recalling that~$\Pi_\eta$ is the projector onto~$\left(\begin{array}{c}
	 \eta  \\0
		\end{array}\right)  $ remark that
$$
\Pi_\eta  b = \sum_{k \in \Z} (b | \Phi_k^-) \Phi_k^- \quad \mbox{and} \quad \Pi_\eta^\perp b  = \sum_{k \in \Z} (b | \Phi_k^+) \Phi_k^+ \, .
$$

		\subsection{Reduction to a finite dimensional setting}
In order to build up our normal form, we need to reduce to a finite dimensional setting, {\it cf} the discussion in paragraph \ref{sec_strategy}. Roughly, Proposition \ref{estimateC2} will help us to get rid of the infinite dimensional subspace 
$$\text{vect}(\{ \Phi_k^{-}, \quad k \in \Z \})  \subset \ker A \, .$$ 
 The point is then  to be able to restrict to $\text{vect}(\{ \Phi^{k}_+, \: |k| \le N)\}$, for some (possibly large) $N$. We first define spectral projectors on low and high modes of~$A$:
	$$
	\Pi_N^\flat b := \sum_{k= -2N}^{2N+1} (b |  \Phi_k^+)  =  \sum_{|k| \le N} \left( (b|\Phi_{2k}^+) \Phi_{2k}^+ + (b|\Phi_{2k+1}^+) \Phi_{2k+1}^+ \right) \, , \quad \Pi_N^\sharp:= \Pi_\eta^\perp- \Pi_N^\flat \, .
	$$
Let us collect some properties of these spectral projectors.
	\begin{lem}\label{lowmodeslemma}
For any  divergence free vector field in the sense of~{\rm(\ref{divfreexih})}~$b$ in~$L^2(\T)^3$, the following holds for all integers~$n$, under Assumption~{\rm(\ref{nonconstant2})}:
$$
\|\partial_{3}^n \Pi_N^\flat b \|_{L^2(\T)} \lesssim \frac{N^{n+\frac{ 1}2} }{\delta } \|\Pi_N^\flat b \|_{L^2(\T)}  \, , \quad  \|\Pi_N^\flat \partial_3^n \Pi_N^\sharp b \|_{L^2(\T)} \lesssim  \frac{N^{n+\frac{ 1}2} }{\delta }\| \Pi_N^\sharp b\|_{L^2(\T)}  \, ,$$ 
and 
$$  \| \Pi_N^\flat \pa_3^n b \|_{L^2(\T)} \lesssim    \frac{N^{n+\frac{ 1}2} }{\delta }  \|b\|_{L^2(\T)}  \, .
$$
\end{lem}
\begin{proof}  The third relation follows easily from the first two and from the inequality 
\begin{align*}
\| \Pi_N^\flat \pa_3^n b \|_{L^2(\T)}& \le \| \Pi_N^\flat  \pa_3^n \Pi_N^\flat b \|_{L^2(\T)} + \|Â \Pi_N^\flat \pa_3^n \Pi_N^\sharp b \|_{L^2(\T)} \\
& \le \|Â \pa_3^n \Pi_N^\flat b \|_{L^2(\T)} + \| \Pi_N^\flat \pa_3^n \Pi_N^\sharp b \|_{L^2(\T)} \, .
\end{align*}
To fix ideas, we assume that $N$ is odd.  To prove the first inequality, we write  
\begin{align*}
 \pa_3 \Pi_N^\flat b & =  \sum_{k=-2N}^{2N+1} (b | \Phi_k^+) \pa_3 \Phi_k^+ \\
  &= \sum_{|k| \le N}  \Bigl(  (b | \Phi_{2k}^+) \pa_3 \Phi_{2k}^+ \: +  \:  (b | \Phi_{2k+1}^+) \pa_3 \Phi_{2k+1}^+ \Bigr) \, .
 \end{align*}
 	Recalling~(\ref{defeklambdak}) we know that  
		\begin{equation}\label{computed3ek}
\partial_3 e_k(x_3)  =\frac{1}{\| \beta_\eta\|_{L^2(\T)}} \big( 	\beta_\eta'(x_3)+ 	\frac1{\mu_k}\beta_\eta^3(x_3)	  \big)
\exp \Big(
		\frac1{\mu_k} \int_0^{x_3} \beta_\eta^2 (y_3) \, dy_3
 		\Big)
		\end{equation}
		and thanks to~{\rm(\ref{nonconstant2})} we have
		$$
	\forall |k| \leq N \, , \quad 		\frac1{|\mu_k|} \lesssim \frac N {\delta  }\,\cdotp
		$$
		It follows that
		$$
			\forall |k| \leq N \, , \quad 	\| 	\partial_3 e_k \|_{L^2(\T)}  \lesssim \frac N {\delta  }
		$$
		so by the Cauchy-Schwarz inequality
		$$
		\|\partial_{3}  \Pi_N^\flat b \|_{L^2(\T)} \lesssim \frac  {N^{\frac{3}2}} {\delta }  \|\Pi_N^\flat b\|_{L^2(\T)} \, .
		$$
		The argument is identical for higher derivatives.
		
		\medskip
		\noindent
		For the second inequality, we write
		 \begin{equation}\label{formulapind3}
		\begin{aligned}
		 \Pi_N^\flat \partial_3 \Pi_N^\sharp b &=  \sum_{j=-2N}^{2N+1}\sum_{k \not\in [-2N,2N+1]}
		  (b | \Phi_k^{+} ) (\partial_3 \Phi_k^{+} |\Phi_j^{+} )\Phi_j ^{+} \\
		 &  = - \sum_{j=-2N}^{2N+1}\sum_{k \not\in [-2N,2N+1]}\  (b |\Phi_k^{+}) ( \Phi_k ^{+}| \pa_3 \Phi_j^{+})\Phi_j^{+} \, \\
		 & = - \sum_{|j| \le N} \sum_{|k| > N} \Bigl( (b |\Phi_{2k}^{+})  (e_k| \pa_3 e_j) \Phi_{2j}^{+}  + (b |\Phi_{2k+1}^{+})  (e_k| \pa_3 e_j) \Phi_{2j+1}^{+} \Bigr) \, ,
		\end{aligned}
		 \end{equation}
		 noticing that by construction~$( \Phi_{2k}^{+} | \Phi_{2kj+1}^{+}) = 0$.

		\noindent  By the definitions \eqref{defe0} and \eqref{defeklambdak}, setting $\mu_0 = +\infty$ (that is $\frac{1}{\mu_0} = 0$), we get
		\begin{equation}\label{ekej}
		\begin{aligned}
 	(e_k |\partial_3e_j) & = \frac{1}{\|  \beta_\eta\|_{L^2(\T)}^2}
		\int_{\T} \beta_\eta(x_3)\big( 	\beta_\eta'(x_3)+ 	\frac1{\mu_j}\beta_\eta^3(x_3)	  \big) \\
		& \qquad \times
\exp \Big(
		\big(\frac1{\mu_k} -\frac1{\mu_j} \big)\int_0^{x_3} \beta_\eta^2 (y_3) \, dy_3
 		 \Big) dx_3 \nonumber \\
\nonumber		& = \frac{1}{\|  \beta_\eta\|_{L^2(\T)}^2}
		\int_{\T} \beta_\eta(x_3)\big( 	\beta_\eta'(x_3)+ 	\frac1{\mu_j}\beta_\eta^3(x_3)	  \big) \\
		& \qquad \times
\exp \Big(2i\pi(k-j)  \frac{1}{\|  \beta_\eta\|_{L^2(\T)}^2}
\int_0^{x_3} \beta_\eta^2 (y_3) \, dy_3 
 		\Big)  dx_3\nonumber \\
		&   = \frac{1}{\|  \beta_\eta\|_{L^2(\T)}} ( \beta_\eta' + \frac{1}{\mu_j} \beta_\eta^3 \, | \,  e_{j-k} ) \,.
		\end{aligned}
		\end{equation}	 
		 Hence
		 \begin{align*}
\|  \Pi_N ^\flat \partial_3 \Pi_N^\sharp b \|_{L^2(\T)}^2 & \le \frac{1}{\|  \beta_\eta\|_{L^2(\T)}^2}  \sum_{|j| \le N}  \Bigl| \sum_{|k| > N}  (b |\Phi_{2k}^{+})  ( \beta_\eta' + \frac{1}{\mu_j} \beta_\eta^3 \, | \,  e_{j-k} ) \Bigr|^2 \\ 
& +  \frac{1}{\|  \beta_\eta\|_{L^2(\T)}^2}  \sum_{|j| \le N}   \Bigl| \sum_{|k| > N}  (b |\Phi_{2k+1}^{+})  ( \beta_\eta' + \frac{1}{\mu_j} \beta_\eta^3 \, | \,  e_{j-k} ) \Bigr|^2 \\
&  \le  \frac{1}{\|  \beta_\eta\|_{L^2(\T)}^2}  \sum_{|j| \le N}    \|  \beta_\eta' + \frac{1}{\mu_j} \beta_\eta^3 \big\|_{L^2(\T)}^2  \\
& \quad \times \Big(  \big\|  \sum_{|k| > N}   (b |\Phi_{2k}^+)  e_{j-k} \big\|_{L^2(\T)}^2 +  \big\|   \sum_{|k| > N}   (b |\Phi_{2k+1}^+)  e_{j-k} \big\|_{L^2(\T)}^2 \Big) \end{align*}
so finally
$$
 		\begin{aligned}
		\|  \Pi_N ^\flat \partial_3 \Pi_N^\sharp b \|_{L^2(\T)}^2& \le   \frac{1}{\|  \beta_\eta\|_{L^2(\T)}^2}  \sum_{|j| \le N} \big \| \beta_\eta' + \frac{1}{\mu_j} \beta_\eta^3\big \|_{L^2(\T)}^2 \,     \sum_{k \not\in [-2N,2N+1]}  | (b | \Phi_k^+) |^2 \\
& \le  \frac{C N^3}{ \delta^2 }  \| \Pi_N^\sharp b \|_{L^2(\T)}^2 \, .
\end{aligned}$$
The argument is identical for higher order derivatives. 		 
\end{proof}
\begin{rmk} 
The decay of the scalar product $(e_k |\partial_3e_j)$ as $|k-j|$ goes to infinity could be specified thanks to stationary phase theorems. For instance, the term $\frac{1}{\|\beta_\eta\|_{L^2}} (\beta_\eta' | e_{k-j})$ in the right hand-side of \eqref{ekej} is proportional to an integral of the form 
$$
	\int_{\T}\Psi''(x_3)\exp \big( i (k-j)  \Psi  (x_3) 
 		\big)  dx_3
	$$
	with~$ \Psi (x_3) :=\displaystyle \int_0^{x_3} \beta_\eta^2(y_3) \, dy_3$. The behaviour of this integral depends on the stationary points of~$\Phi$. For instance, if  $\beta_\eta$  does not vanish, the integral will behave like~$|k-j|^{-n}$ for all~$n$, because $\Phi$ has no stationary point. {\it A contrario}, if $\beta_\eta$ has a (say single) zero of order $m$, then~$\Phi$ has a critical point of order $2m$, and   then, according to~\cite{Stein}, Chapter VIII.1.3,  one has
$$ \left | \int_{\T}\Psi''(x_3)\exp \big( i (k-j)  \Psi  (x_3) 
 		\big)  dx_3\right| \lesssim C |k-j|^{-\frac{2m}{2m+1}}.
$$
\end{rmk}

\medskip
The key proposition to be able to neglect the  high modes  is the following: 
\begin{lem}\label{highmodeslemma}
For any  divergence-free vector field in the sense of~{\rm(\ref{divfreexih})}~$b$ in~$L^2(\T^3)$, the following holds:
$$ \left| \, (C \Pi_\eta^\perp b | \Pi_\eta^\perp b ) - (C \Pi_N^\flat b | \Pi_N^\flat b ) \right| \le \frac{\epsilon(N) }{\delta^2} \| b \|_{L^2(\T)}^2   $$
where $\epsilon(N)$ goes to zero as $N \rightarrow +\infty$. 
\end{lem}

\begin{proof}
We decompose
\begin{equation*}  
 (C \Pi_\eta^\perp b | \Pi_\eta^\perp b) = (C \Pi_N^\flat b | \Pi_N^\flat b ) +  (C \Pi_\eta^\perp b |  \Pi_N^\sharp b) + (C \Pi_N^\sharp b |  \Pi_N^\flat b) \, .
\end{equation*}
We must show that the last two terms go to zero as $N \rightarrow +\infty$. They are very similar, so    we focus on 
$\displaystyle (C \Pi_\eta^\perp b |  \Pi_N^\sharp b)$.   We first consider the  magnetostrophic part, recalling the decomposition~(\ref{defiC}). We have 
\begin{equation*}
(C_m \Pi_\eta^\perp b |  \Pi_N^\sharp b)  
\: = \:  \frac{1}{2} \sum_{j \in \Z, |k| > N}  \left( ( b |  \Phi^+_{2j}) (b | \Phi^+_{2k})  +  ( b |  \Phi^+_{2j+1}) (b | \Phi^+_{2k+1})  \right) (C_m e_j | e_k) \, .   
\end{equation*}
We write 
\begin{align*} \|\beta_\eta\|_{L^2}^2(C_m e_j | e_k)  \: & = -\frac{1}{2} \int_\T   \left( \pa_3^{-1} \bigl(\beta'_{\eta^\perp} \beta_\eta \exp\bigl(\frac{1}{\mu_j} \int_0^{x_3} \beta_\eta^2\bigr)\bigr) \right) \, \beta_\eta^2 \exp\bigl(-\frac{1}{\mu_k} \int_0^{x_3} \beta_\eta^2\bigr)  \\
&  +   \frac{1}{2} \int_\T \left( \pa_3^{-1} \bigl( \beta_\eta^2 \exp\bigl(\frac{1}{\mu_j} \int_0^{x_3} \beta^2\bigr)\bigr)  \right) \beta'_{\eta^\perp} \beta_\eta      
 \exp\bigl(-\frac{1}{\mu_k} \int_0^{x_3} \beta_\eta^2\bigr) \\
 & =: I^1_{jk} + I^2_{jk} \, .
 \end{align*}
 An integration by parts yields (noticing that~$k \neq 0$ by construction) 
 \begin{align*}
  I^1_{jk}  & =   \frac{1}{2} \int_\T   \beta'_{\eta^\perp} \beta_\eta \exp\bigl(\frac{1}{\mu_j} \int_0^{x_3} \beta_\eta^2\bigr)  \,   \left( \pa_3^{-1} \bigl(\beta_\eta^2 \exp\bigl(-\frac{1}{\mu_k} \int_0^{x_3}  \beta_\eta^2\bigr) \bigr) \right) \\
 & = - \frac{\mu_k}{2} \int_\T  \beta'_{\eta^\perp} \beta_\eta  \exp\bigl(\frac{1}{\mu_j} \int_0^{x_3} \beta_\eta^2\bigr) \, \left(  \exp\bigl(-\frac{1}{\mu_k} \int_0^{x_3}  \beta_\eta^2\bigr)   -   \int_\T  \exp\bigl(-\frac{1}{\mu_k} \int_0^{x_3} \beta_\eta^2\bigr) \right) \\
 &  =- \frac{\mu_k}{2}  ( \beta'_{\eta^\perp} | e_{k-j})  \: + \:  \frac{\mu_k}{2}  \biggl(\int_\T  \exp\bigl(-\frac{1}{\mu_k} \int_0^{x_3} \beta_\eta^2\bigr) \biggr) (\beta'_{\eta^\perp} | e_{-j})  \, 
 \end{align*}
 Similarly, for $j \in \Z^*$, 
  \begin{equation*} 
  I^2_{jk}   =     \frac{\mu_j}{2}  ( \beta'_{\eta^\perp} | e_{k-j})   \: - \:  \frac{\mu_j}{2}  \biggl(\int_\T  \exp\bigl(\frac{1}{\mu_j} \int_0^{x_3} \beta_\eta^2\bigr) \biggr) (\beta'_{\eta^\perp} | e_{k})  \, 
\end{equation*}
This implies
\begin{equation} \label{ineq_bj_bk}
\begin{aligned}  
& \frac{1}{2}  \|\beta_\eta\|_{L^2}^2\left| \sum_{j \in \Z^*, |k| > N}  \left( ( b |  \Phi^+_{2j}) (b | \Phi^+_{2k})  +  ( b |  \Phi^+_{2j+1}) (b | \Phi^+_{2k+1})  \right)(C_m e_j | e_k) \right| \\
&  \le C \, \biggl(  \sum_{j \in \Z^*, |k| > N} \left| \frac{1}{k} -\frac{1}{j} \right| \,   |(\beta'_{\eta^\perp} | e_{k-j})|  \,  b_j \, b_k  
\:  + \:    \sum_{j \in \Z^*, |k| > N} \frac{b_k}{|k|} \,  | (\beta'_{\eta^\perp} | e_{-j}) | \,  b_j \,    \\
&  \qquad\qquad +   \,  \sum_{j \in \Z^*, |k| > N} \frac{ b_j }{|j|}  \, | (\beta'_{\eta^\perp} | e_{k}) |  \,  b_k  \biggr)
\end{aligned}
\end{equation}
with  $b_\ell \: := \:   \left| ( b |  \Phi^+_{2\ell}) \right| \: + \:  \left| ( b |  \Phi^+_{2\ell+1}) \right|$.

\medskip
To treat the first term at the right-hand side of \eqref{ineq_bj_bk}, we split the sum over $j$, distinguishing between $|j| \le \frac{N}{2}$ and $|j| > \frac{N}{2}$. One has on the one hand
\begin{align*}
 &  \sum_{0 < |j| \le \frac{N}{2}, |k| > N} \left| \frac{1}{k} -\frac{1}{j} \right| \,   | (\beta'_{\eta^\perp} | e_{k-j}) | \,  b_j \, b_k  \\
&  \le     \sum_{0 < |j| \le \frac{N}{2}, |k| > N} \frac{b_j}{|j|} \,   |(\beta'_{\eta^\perp} | e_{k-j})|   \, b_k  \\
 & \le   \:  \biggl( \sum_{|k| > N} b_k^2 \biggr)^{1/2}  \, \biggl(\sum_{0 < |j| \le \frac{N}{2}} \frac{b_j}{|j|}  \biggr) \biggl( \sum_{|k'| \ge \frac{N}{2}}   |(\beta'_{\eta^\perp} | e_{k'})|^2 \biggr)^{1/2}
 \end{align*}
 by Young's inequality, so
 $$
  \sum_{0 < |j| \le \frac{N}{2}, |k| > N} \left| \frac{1}{k} -\frac{1}{j} \right| \,   | (\beta'_{\eta^\perp} | e_{k-j}) | \,  b_j \, b_k \le 
C \, \epsilon^\perp(N) \, \biggl(  \sum_{|k| > N} b_k^2 \biggr)^{1/2}    \biggl( \sum_{0 < |j| \le \frac{N}{2}} b_j^2  \biggr)^{1/2} \, ,  
 $$
  with 
 $$ \epsilon^\perp(N) := \biggl( \sum_{|\ell| \ge \frac{N}{2}}  | (\beta'_{\eta^\perp} | e_{\ell})|^2 \biggr)^{1/2}.$$
On the other hand
\begin{align*}
&  \sum_{|j| > \frac{N}{2}, |k| > N} \left| \frac{1}{k} -\frac{1}{j} \right| \,   |(\beta'_{\eta^\perp} | e_{k-j})|  \,  b_j \, b_k \\
& \le   \sum_{|j| > \frac{N}{2}, |k| > N}  \left( \frac{1}{|k|} + \frac{1}{|j|} \right) \,   (\beta'_{\eta^\perp} | e_{k-j})  \,  b_j \, b_k \\
& \le \: C \| \beta'_{\eta^\perp} \|_{L^2(\T)}  \,   \biggl( \,  \biggl( \sum_{|k| > N} \frac{b_k}{|k|} \biggr) \,  \biggl( \sum_{|j| > \frac{N}{2}} |b_j|^2 \biggr)^{1/2} \: + \:  \biggl( \sum_{|j| > \frac{N}{2}} \frac{b_j}{|j|} \biggr) \,  \biggl( \sum_{|k| > N} |b_k|^2 \biggr)^{1/2}   \biggr)  \\ 
&    \: \le \:  \frac{C'}{N^{1/2}}    \sum_{|\ell| > \frac{N}{2}} |b_\ell|^2 \, . 
\end{align*}
 
\medskip
The second term at the right-hand side of \eqref{ineq_bj_bk} is bounded by
\begin{align*}
&  \sum_{j \in \Z^*, |k| > N} \frac{b_k}{|k|} \,  | (\beta'_{\eta^\perp} | e_{-j}) | \,  b_j \\
& \le \: C  \| \beta'_{\eta^\perp} \|_{L^2(\T)}   \biggl( \sum_{|k| \ge \frac{N}{2}} \frac{b_k}{|k|} \biggr) \,  \biggl( \sum_{j \in \Z^*} |b_j|^2 \biggr)^{1/2}  
  \le \:  \frac{C'}{N^{1/2}}    \sum_{\ell \in \Z^*} |b_\ell|^2 \, .
 \end{align*}
 
 \medskip
 The third term at the right-hand side of  \eqref{ineq_bj_bk} is bounded by : 
 \begin{align*}
& \sum_{j \in \Z^*, |k| > N} \frac{ b_j }{|j|}  \, | (\beta'_{\eta^\perp} | e_{k}) |  \,  b_k \\
& \le \:  \left( \sum_{j \in \Z^*} \frac{b_j}{|j|} \right) \,  \biggl( \sum_{|k| \ge \frac{N}{2}}  | (\beta'_{\eta^\perp} | e_{k}) |^2 \biggr)^{1/2}  \,  \biggl( \sum_{|k| \ge \frac{N}{2}} |b_k|^2 \biggr)^{1/2} \le  C \, \epsilon^\perp(N) \sum_{\ell\in \Z^*} |b_\ell|^2.
\end{align*}
Notice that $\sum_{\ell \in \Z} |b_\ell|^2 \le 2 \| b \|^2_{L^2(\T)}$.   Gathering all bounds, we find 
 $$  \frac{1}{2}  \sum_{j \in \Z^*, |k| > N}  \left( ( b |  \Phi^+_{2j}) (b | \Phi^+_{2k})  +  ( b |  \Phi^+_{2j+1}) (b | \Phi^+_{2k+1})  \right)(C_m e_j | e_k) \: \le \: \frac{\epsilon_*(N)}{\delta^2} \| b \|_{L^2(\T)}  $$
 where $\epsilon_*(N)$ goes to zero as $N \rightarrow +\infty$. We still have to examine the case $j=0$. We write 
 \begin{align*}
&  \frac{1}{2} \left| \sum_{|k| > N}  \left( ( b |  \Phi^+_0) (b | \Phi^+_{2k})  +  ( b |  \Phi^+_{1}) (b | \Phi^+_{2k+1})  \right) (C_m e_0 | e_k) \right| \\
&  \le   C \, \| b \|_{L^2(\T)} \, \biggl( \sum_{|k| > N}  b_k^2 \biggr)^{1/2} \, \biggl( \sum_{|k| > N}  |(C_m e_0 | e_k)|^2 \biggr)^{1/2}  \le  \frac{\epsilon_0(N)}{\delta^2} \| b \|_{L^2(\T)}^2
  \end{align*}
with 
$$\epsilon_0(N) :=  \biggl( \sum_{|k| > N}  |(C_m e_0 | e_k) |^2 \biggr)^{1/2} \rightarrow 0 \quad \text{as $N \rightarrow +\infty$ }   .$$
Note that there is a priori no rate of convergence for this term~$\epsilon_0(N)$, contrary to the other terms which could be quantified: this is related  to the fact that~$\beta_\eta e_0$ is not mean free.
Eventually, 
\begin{equation} \label{bound_Cm}
\left| (C_m \Pi_\eta^\perp b |  \Pi_N^\sharp b) \right| \: \le \:  \frac{\epsilon_0(N)}{\delta^2}\,   \| b \|_{L^2(\T)}  
\end{equation}
for some $\epsilon_0(N)$ going to zero as $N$ goes to infinity. 

\medskip
To conclude the proof of the lemma, we still have to study the contribution of the other two terms involved in the definition of the operator $C$, see \eqref{defiC}, namely 
\begin{equation*}
C_{g,1} \, b \:   :=   \: A_m  \left( \frac{\beta_\eta}{\| \beta_\eta \|_{L^2(\T)}^2} \int_\T  i \,  b_3  \, B'  \, dx_3 \right), \quad    C_{g,2} \, b \: =   - \:  \frac{i}{\| \beta_\eta \|_{L^2(\T)}^2} \int_\T \beta_\eta  \left( A_m \,  b \right)_3 \, dx_3 \, B'. 
\end{equation*}
We first consider 
\begin{equation*}
 ( C_{g,1} \Pi_\eta^\perp b | \Pi^\sharp_N b) =  -\frac{1}{\| \beta_\eta \|_{L^2(\T)}^2} \left( \int_\T  i \, (\Pi_\eta^\perp b)_3  \, \beta'_{\eta^\perp} \, dx_3 \right) \left( A_m  \bigl(\beta_\eta \bigl(\begin{smallmatrix} \eta^\perp \\ 0 \end{smallmatrix} \bigr)\bigr) |  \Pi_N^\sharp b\right).   
 \end{equation*}
 It is clear that 
 $$  \left| \frac{1}{\| \beta_\eta \|_{L^2(\T)}^2}\left( \int_\T  i \, (\Pi_\eta^\perp b)_3  \,  \beta'_{\eta^\perp}  \, dx_3 \right) \right| \: \le \: \frac C{\delta^2} \, \| b \|_{L^2(\T)} \, .   $$
 Then, 
 \begin{equation*}
 \left( A_m  \bigl(\beta_\eta \bigl(\begin{smallmatrix} \eta^\perp \\ 0 \end{smallmatrix} \bigr)\bigr) |  \Pi_N^\sharp b\right)
  = \sum_{k \not\in [-2N,2N+1]}   \left( A_m  \bigl(\beta_\eta \bigl(\begin{smallmatrix} \eta^\perp \\ 0 \end{smallmatrix} \bigr) \bigr) | \Phi^+_k \right) \, (b | \Phi^+_k) 
 \end{equation*}
 so that 
\begin{equation*}
\left|\left( A_m  \bigl(\beta_\eta \bigl(\begin{smallmatrix} \eta^\perp \\ 0 \end{smallmatrix} \bigr)\bigr) |  \Pi_N^\sharp b\right)\right| 
\:  \le \: C \,  \delta(N)  \biggl( \sum_{k \not\in [-2N,2N+1]}  | (b | \Phi^+_k)|^2 \biggr)^{1/2} \le C \,  \delta(N) \| b \|_{L^2(\T)}
\end{equation*}
with 
$$ \delta(N) := \Bigl( \sum_{|k| \ge N} \left| \left( A_m  \bigl(\beta_\eta \bigl(\begin{smallmatrix} \eta^\perp \\ 0 \end{smallmatrix} \bigr) \bigr) |  \Phi^+_k\right) \right|^2  \Bigr)\: \xrightarrow[N \rightarrow +\infty]{} 0 \, .$$
Hence: 
$$ | ( C_{g,1} \Pi_\eta^\perp b | \Pi^\sharp_N b) | \: \le \: \delta_1(N) \| b \|^2_{L^2(\T)} $$
for some $\delta_1(N)$ going to zero as $N$ goes to infinity.

\medskip
The treatment of $(C_{g,2} \Pi_\eta^\perp b | \Pi^\sharp_N b)$ follows the same lines as the previous one, and we omit it for brevity: we get 
$$  |(C_{g,2} \Pi_\eta^\perp b | \Pi^\sharp_N b)| \le  \: \epsilon_2(N) \| b \|^2_{L^2(\T)} $$
 for some $\epsilon_2(N)$ going to zero as $N$ goes to infinity.   Putting these last two estimates together with \eqref{bound_Cm} ends the proof of the lemma. 
\end{proof}
\subsection{Construction of the matrix $Q$}
The aim of this paragraph is to prove the following  result.
\begin{Prop} \label{prop_Q}
There exist two  finite dimensional operators $Q,T : \Pi_N^\flat L^2(\T)^3 \rightarrow \Pi_N^\flat L^2(\T)^3$ such that 
$$ \:  [\Pi_N^\flat \, Q \, \Pi_N^\flat, A] \: = \:  - \Pi_N^\flat \, C \, \Pi_N^\flat +\e T \, ,  $$
 where $C_m$ is defined in \eqref{defAmCm} and where
 for all smooth~$b$ satisfying the Taylor constraint,
  \begin{equation}\label{formulaT}
\big | (Tb | b)_{L^2} \big | \le C \|b\|_{L^2}^2 \, .
 \end{equation}
\end{Prop}
\begin{proof}
We are going to construct~$Q$, and  show that
$$
T = \frac 1 \eps \Pi_\eta^\perp \Pi_{e_0} C_m  \Pi_{e_0} \Pi_\eta^\perp \, .
$$ 
We notice indeed that as $\Pi_\eta^\perp$ and $\Pi_{e_0}$ commute, we have 
$$ C_m \Pi_{e_0} \Pi_\eta^\perp b  = C_m \Pi_\eta^\perp \Pi_{e_0} b  \, .$$
Then if~$b$ satisfies  the linearized Taylor constraint, which amounts to identity \eqref{formulaPi0b_bis}, then we get easily  
$$  |( C_m \Pi_{e_0} \Pi_\eta^\perp b |  \Pi_{e_0} \Pi_\eta^\perp b) | \: \le \: C \eps \| b \|_{L^2}^2  \, . $$So now let us prove that
there exists $Q$ satisfying:  for all $|i|, |j| \le N$, 
$$ (AQ \Phi^+_i - Q A \Phi^+_i | \Phi^+_j )  = -  (C \Phi^+_i | \Phi_j^+)  + \left( C_m \Pi_{e_0} \Phi^+_i | \Pi_{e_0} \Phi^+_j\right). $$
Let $(\lambda_k)_{k \in \Z}$  the family of eigenvalues of $A$ associated to $(\Phi^+_k)_{k \in \Z}$.   We recall from Proposition~\ref{spectraldec}  that
$$ \lambda_0  = \lambda_1 = 0 \, , \quad \lambda_{2k} = \mu_k\, , \quad \lambda_{2k+1} = - \mu_k \, , \quad k \in \Z^* \, . $$ 
The last equality reads 
\begin{equation} \label{homology}
 (\lambda_j - \lambda_i) (Q  \Phi^+_i | \Phi^+_j ) = -  (C \Phi^+_i | \Phi_j^+)    + \left( C_m \Pi_{e_0} \Phi^+_i | \Pi_{e_0} \Phi^+_j\right). 
 \end{equation}
A necessary and sufficient condition for the existence of $Q$ is that the right-hand side is zero when $\lambda_i = \lambda_j$. There are three cases to consider:  
\begin{itemize}
\item The first case is of course when $i=j$. We compute 
\begin{align*}
&  \left( C \Phi^+_i | \Phi^+_i \right)  \:  = \:   \left( C_m \Phi^+_i |  \Phi_i^+ \right) \\
 & \: + \left( A_m  \left( \frac{\beta_\eta}{\| \beta_\eta \|_{L^2(\T)}^2} \int_\T  i \,  \Phi^+_{i,3}  \, B'  \, dx_3 \right) \: | \Phi^+_i \right)  \\
& \: -  \left( \frac{i}{\| \beta_\eta \|_{L^2(\T)}^2} \int_\T \beta_\eta  \left( A_m \,  \Phi_i^+ \right)_3 \, dx_3 \, B'  \: |  \Phi^+_i \right)   \, .
\end{align*}
In the case $i=0$ or $i=1$, we have $\Phi^+_i = \Pi_{e_0} \Phi^+_i$. Moreover, we know from \eqref{PiAmPi} that $\Pi_{e_0} A_m \Pi_{e_0} = 0$. It implies that the second term vanishes. Finally, 
$$ \int_\T \beta_\eta  \left( A_m \,  \Phi_i^+ \right)_3 \, dx_3 =  \frac{i}{\sqrt{2}\| \beta_\eta \|_{L^2(\T)}} \int_\T \beta_\eta^2 \pa_3^{-1} \beta_\eta^2 dx_3 = 0 $$
so that 
\begin{equation*}
\mbox{if} \quad i \in \{0,1\} \, , \quad (C \Phi^+_i | \Phi^+_i) \:  = \:  ( C_m \Phi^+_i | \Phi^+_i ) \: = \:  ( C_m \Pi_{e_0} \Phi^+_i | \Pi_{e_0} \Phi^+_i) \, .  
\end{equation*}
This  means that right-hand side of \eqref{homology} is zero. 

In the case $i = 2k$, $k \in \Z^*$, we find 
\begin{align*}
 \left( C_m \Phi^+_{2k} |  \Phi^+_{2k} \right) &  =  - \Re \int_\T    \beta_\eta   \pa_3^{-1} ( \beta'_{\eta^\perp}  \,  e_k )  \overline{e_k}   
  = \Re \int_\T   e_k   \frac{\beta'_{\eta^\perp}}{\beta_\eta} \beta_\eta \pa_3^{-1} ( \beta_\eta \overline{e_k} )  \\
 & =  \Re   \int_\T  e_k \frac{\beta'_{\eta^\perp}}{\beta_\eta}  \Pi_{e_0}^\perp \beta_\eta   \pa_3^{-1}(  \beta_\eta \Pi_{e_0}^\perp \overline{e_k} ) 
  + \Re   \int_\T  e_k \frac{\beta'_{\eta^\perp}}{\beta_\eta}  \Pi_{e_0} \beta_\eta    \pa_3^{-1} ( \beta_\eta \overline{e_k} ) \\
& =  \Re  \: \overline{\mu_k}    \int_\T \frac{\beta'_{\eta^\perp}}{\beta_\eta}  |e_k|^2 \:   + \:  \Re  \frac{1}{\| \beta \|_{L^2(\T)}^2} \int_\T \beta'_{\eta^\perp} e_k dx_3 \, \int_\T \beta^2 \pa_3^{-1} \beta_\eta \overline{e_k} \\
& =   \Re  \frac{1}{\| \beta \|_{L^2(\T)}^2} \int_\T \beta'_{\eta^\perp} e_k dx_3 \, \int_\T \beta^2 \pa_3^{-1} \beta_\eta \overline{e_k}   \, .
\end{align*}
As regards the second term,
\begin{align*}
& \left( A_m  \left( \frac{\beta_\eta}{\| \beta_\eta \|_{L^2(\T)}^2} \int_\T  i \,  \Phi^+_{i,3}  \, B'  \, dx_3 \right) \: | \Phi^+_i \right) \\
& = \frac{1}{\| \beta_\eta \|_{L^2(\T)}^2} \int_\T \frac{\beta'_{\eta^\perp}}{\sqrt{2}} e_k \, dx_3 \, \int_\T  \frac{\beta_\eta}{\sqrt{2}} \pa_3^{-1} \beta_\eta^2 \overline{e_k} \, dx_3 \\
& = \frac{-1}{\| \beta_\eta \|_{L^2(\T)}^2} \int_\T \frac{\beta'_{\eta^\perp}}{\sqrt{2}} e_k \, dx_3 \, \int_\T  \frac{\beta_\eta^2}{\sqrt{2}} \pa_3^{-1} \beta_\eta \overline{e_k} \, dx_3  \, .
\end{align*}
As regards the third term,
\begin{align*}
-  \left( \frac{i}{\| \beta_\eta \|_{L^2(\T)}^2} \int_\T \beta_\eta  \left( A_m \,  \Phi_{2k}^+ \right)_3 \, dx_3 \, B'  \: |  \Phi^+_{2k} \right) \\
=  \frac{-1}{\| \beta_\eta \|_{L^2(\T)}^2}   \int_\T \frac{\beta_\eta^2}{\sqrt{2}} \pa_3^{-1} \beta_\eta e_k \, dx_3 \, \int_\T  \frac{\beta'_{\eta^\perp}}{\sqrt{2}} \overline{e_k} \, dx_3 \, . 
\end{align*}
We find $ (C \Phi^+_{2k} | \Phi^+_{2k})  = 0$ as expected. In the same way, $ (C \Phi_{2k+1}^+ | \Phi_{2k+1}^+)  = 0$. 

\medskip
\item the second case is when  $(i,j) \in \{(0,1), (1,0)\}$, which implies $\lambda_0 = \lambda_1 = 0$. We compute
\begin{align*}
 & \left( C \Phi^+_0 | \Phi^+_1 \right)  \:  = \:   \left( C_m \Phi^+_0 |  \Phi_1^+ \right) \\
 & \: + \left( A_m  \left( \frac{\beta_\eta}{\| \beta_\eta \|_{L^2(\T)}^2} \int_\T  i \,  \Phi^+_{0,3}  \, B'  \, dx_3 \right) \: | \Phi^+_1 \right)  \\
& \: -  \left( \frac{i}{\| \beta_\eta \|_{L^2(\T)}^2} \int_\T \beta_\eta  \left( A_m \,  \Phi_0^+ \right)_3 \, dx_3 \, B'  \: |  \Phi^+_1 \right)   \, .
\end{align*}
We argue exactly as in the first case to conclude that 
\begin{equation*}
(C \Phi^+_0 | \Phi^1_+) \:  = \:  ( C_m \Phi^+_0 | \Phi^+_1 ) \: = \:  ( C_m \Pi_{e_0} \Phi^+_0 | \Pi_{e_0} \Phi^+_1) 
\end{equation*}
which means that  the right-hand side of \eqref{homology} is zero when $i=0, j=1$.  By the same computation, it is also true when $i=1,j=0$. 

\medskip
\item the last case is when  $(i,j) \in \{(2k,-2k+1), (-2k+1,2k)\}$  for some $k \in \Z^*$, which implies  $\lambda_i = \lambda_j = \mu_k$. 
Note that 
$$\Phi^+_{2k} = \frac 1 {\sqrt 2}e_k \left( \begin{smallmatrix} \eta^\perp \\ i \end{smallmatrix} \right), \quad \Phi^+_{-2k+1} = \overline{\Phi_{2k}^+} =  \frac 1 {\sqrt 2}e_{-k} \left( \begin{smallmatrix} \eta^\perp \\ -i \end{smallmatrix} \right), \quad \text{with} \quad  \left( \left( \begin{smallmatrix}\eta^\perp \\ i \end{smallmatrix} \right) | \left( \begin{smallmatrix} \eta^\perp \\ -i \end{smallmatrix} \right) \right)= 0 \, . $$
This last orthogonality property implies easily that 
\begin{equation*}
\left( C_m \Phi_{2k}^+ | \Phi_{-2k+1}^+ \right) = 0 \, .
\end{equation*}
Then,  
 \begin{align*}
& \left( A_m  \left( \frac{\beta_\eta}{\| \beta_\eta \|_{L^2(\T)}^2} \int_\T  i \,  \Phi^+_{2k,3}  \, B'  \, dx_3 \right) \: | \Phi^+_{-2k+1} \right) \\
& \quad\quad = \frac{1}{\| \beta_\eta \|_{L^2(\T)}^2} \int_\T \frac{\beta'_{\eta^\perp}}{\sqrt{2}} e_k \, dx_3 \, \int_\T  \frac{\beta_\eta^2}{\sqrt{2}} \pa_3^{-1} \beta_\eta e_k \, dx_3 
\end{align*}
while 
\begin{align*}
-  \left( \frac{i}{\| \beta_\eta \|_{L^2(\T)}^2} \int_\T \beta_\eta  \left( A_m \,  \Phi_{2k}^+ \right)_3 \, dx_3 \, B'  \: |  \Phi^+_{-2k+1} \right) \\
=  \frac{-1}{\| \beta_\eta \|_{L^2(\T)}^2}   \int_\T \frac{\beta_\eta^2}{\sqrt{2}} \pa_3^{-1} \beta_\eta e_k \, dx_3 \, \int_\T  \frac{\beta'_{\eta^\perp}}{\sqrt{2}} e_k \, dx_3 \, . 
\end{align*}
Finally, $(C \Phi_{2k}^+ | \Phi_{-2k+1}^+)  = 0$.  In the same way, $(C \Phi_{-2k+1}^+ | \Phi_{2k}^+)  = 0$. 
\end{itemize}
This ends the proof of the proposition. 
\end{proof}
\begin{rmk} From the relation \eqref{homology}, one infers that 
\begin{equation} \label{estimQ}
 \left|(Q \Phi^+_i | \Phi^+_j)\right| \: \le \: C N^2, \quad \forall |i|, |j| \le N
 \end{equation}
From this estimate, we deduce that the operator norm  of $Q$  from $\Pi_N^\flat L^2$ to itself is bounded by:  
\begin{equation}\label{estimQ_2}
\| Q \| \: \le \: \left( \sum_{ij}  \left|(Q \Phi^+_i | \Phi^+_j)\right|^2 \right)^{1/2} \le C \, N^3\,. 
\end{equation}
\end{rmk}

\subsection{Conclusion}
We have now all the elements to establish the stability estimate \eqref{theresult}.
Following the discussion of Paragraph \ref{sec_strategy}, we introduce the new unknown 
$$ d \: := \: ({\rm{Id}} + \eps Q_N ) b \, , \quad Q_N \: := \:  \Pi_N^\flat Q \Pi_N^\flat\,,  $$ 
with $Q$ constructed in Proposition \ref{prop_Q}. From the estimate \eqref{estimQ_2}, we deduce that  there exists~$C > 0$ such that 
$$ \eps \| Q_N \| \: \le \: C \eps N^3 \,.$$  
Hence, for $\eps N^3 \ll 1$, the operator $({\rm{Id}} + \eps Q_N)$ is invertible, with  
$$ ({\rm{Id}}  + \eps Q_N)^{-1} \: = \:{\rm{Id}}  - \eps Q_N + O(\eps^2 N^6) \, . $$
It follows that there exists $C > 0$ such that 
\begin{equation} \label{estimatebd} 
(1 - C\eps N^3) \| b\|_{L^2(\T)} \: \le \: \| d \|_{L^2(\T)} \: \le \: (1 + C\eps N^3) \| b\|_{L^2(\T)} \, . 
\end{equation}
As $b$ satisfies equation \eqref{induction3}, we obtain
\begin{equation} \label{induction_d}
\begin{aligned}
\pa_t d & = \frac{A d}{\eps^3} + \frac{([Q_N,A] +  C) d}{\eps^2} + \frac{R d}{\eps^2} - \frac{1}{\eps^2} d + \pa_3^2 d  \\
& \quad + O(\frac{N^3}{\eps}) d  + \eps [Q_N, \pa_3^2] ({\rm{Id}}  + \eps Q_N)^{-1} d \, , 
\end{aligned}
\end{equation}
where $O(\frac{N^3}{\eps})$ stands for an operator whose operator norm is controlled by $C \, \frac{N^3}{\eps}$. 
As regards the commutator $[Q_n, \pa_3^2]$, we invoke Lemma \ref{lowmodeslemma} together with \eqref{estimQ_2} to get 
$$   \| [Q_N, \pa_3^2] \| \: \le \: C N^3 N^{5/2} \: = \:  C \,  N^{11/2}\, . $$
Thanks to Proposition \ref{prop_Q}, we get 
\begin{align*} \pa_t d & = \frac{A d}{\eps^3} + \frac{ (C  +\eps T -  \Pi_N^\flat C \Pi_N^\flat ) d}{\eps^2} + \frac{R d}{\eps^2} - \frac{1}{\eps^2} d + \pa_3^2 d  \\
& \quad + O\bigl(\frac{N^3}{\eps} + \eps N^{11/2}\bigr) d   \, .
\end{align*}
We  then multiply by $d$ and perform an energy estimate: 
\begin{multline}
\frac{1}{2}\frac{d}{dt} \| d \|^2_{L^2(\T)} + \frac{1}{\eps^2}  \| d \|^2_{L^2(\T)}\: + \: \| \pa_3 d \|^2 _{L^2(\T)}\: \le \: \frac{1}{\eps^2} \Re \, ((Cd | d) - (C \Pi_N^\flat d | \Pi_N^\flat d))  \\ 
+  \frac{1}{\eps} \Re \, (T d   |   d)  \: + \:  \frac{1}{\eps^2} \Re \, ( R d   |  d)   + C \, \left( \frac{N^3}{\eps} + \eps N^{11/2}\right) \| d \|_{L^2(\T)}^2\, . 
\end{multline}
Equivalently 
\begin{multline}
\frac{1}{2}\frac{d}{dt} \| d \|^2_{L^2(\T)} + \frac{1}{\eps^2}  \| d \|^2_{L^2(\T)}  \: + \: \| \pa_3 d \|^2 _{L^2(\T)}\: \le \: \frac{1}{\eps^2} \Re \, ((Cd | d) - (C \Pi_N^\flat d | \Pi_N^\flat d))  \\ 
+  \frac{1}{\eps} \Re \, (T b   |   b) \: + \:  \frac{1}{\eps^2} \Re \, ( R b   |  b)   + C \, \left( \frac{N^3}{\eps} + \eps N^{11/2}\right) \| d \|_{L^2(\T)}^2\, . 
\end{multline}
By \eqref{Rbb} 
$$     | (R b | b ) | \le C  \, \eps \| b\|_{L^2(\T)}^2 \le C' \eps  \| d \|_{L^2(\T)}^2$$
and by \eqref{formulaT}
$$     | (T b | b ) | \le C   \| b\|_{L^2(\T)}^2 \le C'   \| d \|_{L^2(\T)}^2 \, .$$
 Thus, the energy estimate on $d$ reduces to 
 $$ \begin{aligned}
 \frac{1}{2}\frac{d}{dt} \| d \|^2_{L^2(\T)} +\frac{1}{\eps^2}  \| d \|^2_{L^2(\T)} \: + \: \| \pa_3 d \|^2 _{L^2(\T)}\: \le \: \frac{1}{\eps^2} \Re \, ((Cd | d) - (C \Pi_N^\flat d | \Pi_N^\flat d))  \\
 + \: C \, \bigl( \frac{N^3}{\eps} + N^{11/2}\bigr) \| d \|_{L^2(\T)}^2\, . 
 \end{aligned}
 $$
Eventually, to estimate the first term at the right-hand side, we combine Proposition  \ref{estimateC2}  and Lemma \ref{highmodeslemma} to get 
$$ \left|  (Cb | b) - (C \Pi_N^\flat b | \Pi_N^\flat b) \right| \: \le \:  \left( C \eps + \delta(N) \right) \| b \|_{L^2(\T)}^2 $$
which gives 
$$ \left|  (Cd | d) - (C \Pi_N^\flat d | \Pi_N^\flat d) \right| \: \le \:  \left( C \eps + \delta(N) + \eps N^3 \right) \| d \|_{L^2(\T)}^2\, . $$  
We end up with 
$$ \frac{1}{2}\frac{d}{dt} \| d \|^2_{L^2(\T)} + \frac{\| d \|^2_{L^2(\T)}}{\eps^2}  \: + \: \| \pa_3 d \|^2 _{L^2(\T)}\: \le  C_0  \left( \frac{\delta(N)}{\eps^2} + \frac{N^3}{\eps} + \eps N^{11/2} \right) \| d \|_{L^2}^2\, .  $$
Taking $N$ large enough so that $C_0 \delta(N) \le \frac{1}{2}$, one can absorb the first term at the right-hand side by the diffusion term at the left-hand side. This $N$ being fixed, for small enough $\eps$, all remaining terms can be absorbed as well.  This leads to the estimate
$$
  \| d(t) \|^2_{L^2(\T)} + \frac 1{\eps^2}\int_0^t {\| d(t') \|^2_{L^2(\T)}}  \: dt' +\int_0^t \| \pa_3 d(t') \|^2 _{L^2(\T)}\: dt'\le  C  \| d (0)\|_{L^2}^2\, .
$$
Using~(\ref{estimatebd}) this implies that
$$
  \| b(t) \|^2_{L^2(\T)} + \frac 1{\eps^2}\int_0^t {\| b(t') \|^2_{L^2(\T)}}  \: dt' +\int_0^t \| \pa_3 d(t') \|^2 _{L^2(\T)}\: dt'\le  C  \| b(0)\|_{L^2}^2\, ,
$$
and finally since  Lemma \ref{lowmodeslemma} together with \eqref{estimQ_2}  imply that
$$   \| [Q_N, \pa_3] \| \: \le \: C N^3 N^{3/2} \: = \:  C \,  N^{9/2}\, , $$
we have
$$
\|\pa_3 b\|_{L^2(\T)} \le  \|\pa_3 d\|_{L^2(\T)} + C\eps N^{9/2}\| d\|_{L^2(\T)} 
$$
so for~$\eps$ small enough,
$$
  \| b(t) \|^2_{L^2(\T)} + \frac 1{\eps^2}\int_0^t {\| b(t') \|^2_{L^2(\T)}}  \: dt' +\int_0^t \| \pa_3 b(t') \|^2 _{L^2(\T)}\: dt'\le  C  \| b(0)\|_{L^2}^2\, .
$$
We notice that there is a global control on~$b$ in~$L^2$, with no exponential loss in~$t$. The exponential appearing on the right-hand side of~(\ref{theresult}) in Theorem~\ref{mainthm} is due to the contrinbution of low horizontal frequencies as explained in Section~\ref{reductionlargehorizontal}. The end of the proof  of the theorem consists in noticing that the velocity is obtained as a second order operator with respect to~$b$ (see Sections~\ref{computationmagnetostropic} and~\ref{computationgeostropic}).
Theorem~\ref{mainthm} is proved.
  
\section*{Acknowledgements}   
  The authors are very grateful to Laure Saint-Raymond for multiple discussions at the early stage of this work. They acknowledge the support of ANR Project Dyficolti ANR-13-BS01-0003-01.

\end{document}